\documentclass[11pt]{amsart}
 \usepackage{amsmath,amsthm,amsfonts,amssymb}
\usepackage{verbatim}
\usepackage{latexsym}
\newtheorem{theorem}{Theorem}[section]
\newtheorem{lemma}[theorem]{Lemma}

\newtheorem{definition}[theorem]{Definition}
\newtheorem{corollary}[theorem]{Corollary}

\newtheorem{remark}[theorem]{Remark}

\begin{document}
\title{A classification of anisotropic Besov spaces}
\author{Jahangir Cheshmavar}
\author{Hartmut F\"uhr}
\email{j$_{_-}$cheshmavar@pnu.ac.ir}
\email{fuehr@matha.rwth-aachen.de}
\address{Department of Mathematics, Payame Noor University, P.O. Box: 19395-3697, Tehran, Iran}
\address{Lehrstuhl A f\"ur Mathematik, RWTH Aachen University, D-52056 Aachen, Germany}
\maketitle

\begin{abstract}
We study (homogeneous and inhomogeneous) anisotropic Besov spaces associated to expansive dilation matrices $A \in {\rm GL}(d,\mathbb{R})$, with the goal of clarifying when two such matrices induce the same scale of Besov spaces. For this purpose, we first establish that anisotropic Besov spaces have an alternative description as decomposition spaces. This result allows to relate properties of function spaces to combinatorial properties of the underlying coverings. This principle is applied to the question of classifying dilation matrices. It turns out the scales of homogeneous and inhomogeneous Besov spaces differ in the way they depend on the dilation matrix: Two matrices $A,B$ that induce the same scale of homogeneous Besov spaces also induce the same scale of inhomogeneous spaces, but the converse of this statement is generally false. Furthermore, the question whether $A,B$ induce the same scale of homogeneous spaces is closely related to the question whether they induce the same scale of Hardy spaces; 
the latter 
question had been previously studied by Bownik. We give a complete characterization of the different types of equivalence in terms of the Jordan normal forms of $A,B$.
\end{abstract}



\section{Introduction}\label{introduction}

Let $A \in \mathbb{R}^{d \times d}$ denote a matrix whose eigenvalues all have modulus $>1$. Matrices of this kind, often with additional properties (such as integer entries) were the basis of the study of discrete wavelet systems (frames or bases) obtained from dilations by powers of $A$ and suitable translations. Here, the initial choice was to take $A = 2 \cdot I_d$, but it was soon recognized that many more matrices $A$ could be employed to construct multiresolution analyses, wavelet frames and bases, see e.g. \cite{GroMa,Str,LaWa,BaMe,BoSp}.

Anisotropic Besov spaces associated to diagonal, anisotropic scaling matrices have been studied since the work of Besov, Il$'$in and Nikol$'$ski{\u\i} \cite{Besov_et_al}, see for example Schmeisser and Triebel \cite{Schmeisser_Triebel}, Triebel \cite{Triebel_FSI,Triebel_FSII,Triebel_04}, Dintelmann \cite{Dintelmann_95}, Farkas \cite{Farkas}, Hochmuth \cite{Hochmuth}, Garrig\'os, Hochmuth and Tabacco \cite{Garrigos} and Kyriazis \cite{kyr}. This class of function spaces was further extended by Bownik in 2005 \cite{Bow05}, to allow arbitrary expansive dilation matrices. Bownik showed that many of the well-known results concerning the relationship between wavelet bases and (isotropic) Besov spaces carries over to the anisotropic setting. Further work in this direction can be found in \cite{BaBe1,BaBe2,LiBaBoYaYu,CaMoRo}.

Much of the existing literature is concerned with alternative descriptions of the spaces, say in terms of moduli of continuity, atomic decompositions, etc. A recurring theme connected to these questions is a certain robustness of the various descriptions. The existing results often provide an understanding which variations of the known criteria result in the same spaces, for instance by prescribing necessary and/or sufficient conditions on the atoms that allow to characterize the function spaces. This paper is intended as a further contribution to this discussion, by studying the question which matrices induce the same scale of Besov spaces. For the related case of Hardy spaces, this question had already been studied by Bownik in \cite{Bow03}. As will be seen below, it can be shown that two expansive matrices define the same scale of homogeneous anisotropic Besov spaces if and only if they induce the same scale of anisotropic Hardy spaces. Hence the results of the cited paper are  directly relevant to our 
paper. Note however 
that one important result, namely \cite[Theorem (10.3)]{Bow03}, which provides a characterization of this property in terms of generalized eigenspaces, is incorrect; a counterexample can be found in Remark \ref{rem:counter_bownik} below. Hence our results  provide a complement and partial correction to \cite{Bow03}.

\subsection{Overview of the paper}
In terms of technique, our paper relies mostly on Bownik's papers \cite{Bow03,Bow05}, as well as on the recent work by Voigtlaender on decomposition spaces \cite{VoDiss,Vo_Embed1}. 
It is structured as follows: Sections \ref{sect:ani_besov} to \ref{sect:exp_matr} are mostly introductory. We review the basic definitions pertaining to anisotropic Besov spaces and expansive matrices. In particular, we introduce the notion of homogeneous quasi-norms associated to an expansive matrix, and recall their basic properties. We also introduce decomposition spaces, which will be an important tool in the subsequent arguments. Decomposition spaces are a flexible construction of function spaces due to Feichtinger and Gr\"obner \cite{DecompositionSpaces1}, which are based on certain coverings of the frequency space (or a subset thereof). The main advantage of decomposition spaces is that they translate the problem of comparing of decomposition spaces associated to different coverings to the task of comparing the coverings themselves, via Theorem \ref{thm:rigidity} and Lemma \ref{lem:suf_dc_equal}. 

Section \ref{sect:ani_dec} then contains the first important new result, the alternative characterization of anisotropic Besov spaces as (Fourier side) decomposition spaces. The result as such is not surprising, and has already  been obtained for several special cases: For the isotropic setting, it was proved in \cite{DecompositionSpaces1} for the inhomogeneous case, and proved also for the homogeneous case in \cite{VoDiss}. For anisotropic inhomogeneous spaces with diagonal dilation matrix, it was observed in \cite{BorupNielsenDecomposition}. The theorem for the general case seems to be new, however. As a first consequence of this observation, we show a rigidity result for Besov spaces, see Theorem \ref{thm:rigidity_besov}. For the coverings used in the decomposition space description of anisotropic Besov spaces induced by an expansive matrix $A$, one can employ annuli with respect to an $A^T$-homogeneous quasi-norms $\rho_{A^T}$; here $A^T$ denotes the transpose of $A$. We use this observation to translate 
the question whether two matrices $A$ and $B$ yield the same anisotropic Besov spaces to a question about the 
relationship between the associated quasi-norms $\rho_{A^T}$ and $\rho_{B^T}$. 
For the homogeneous case, the induced spaces are equal if and only if $\rho_A^T$ and $\rho_B^T$ are equivalent (in the usual sense); 
see Lemma \ref{lem:char_equiv_matr}. For the inhomogeneous case, equality of the induced Besov spaces holds if and only if the quasi-norms are {\em coarsely equivalent},  see Lemma \ref{lem:char_equiv_matr_ih}, and confer to Definition \ref{defn:coarse} for coarse equivalence. As a corollary we get that $A$ and $B$ induce the same inhomogeneous spaces if they induce the same homogeneous ones. Furthermore, the results from \cite{Bow03} yield that the homogeneous case is equivalent to the question whether $A$ and $B$ induce the same anisotropic Hardy spaces, by Remark \ref{rem:rel_Hardy_eq}. 

Thus the discussion whether two expansive matrices $A$ and $B$ induce the same scales of anisotropic Besov spaces is reduced to that of (coarse) equivalence of associated homogeneous quasi-norms. Section \ref{sect:char_equiv} is devoted to completely clarifying this question. We show that for each expansive matrix $A$ there exists an expansive matrix $A'$ such that $A'$ has only positive eigenvalues, and $|{\rm det}(A')|=2$, and the norms induced by the two matrices are equivalent. We call $A'$ the {\bf expansive normal form} of $A$. The significance of this notion is provided by Theorem \ref{thm:class_equiv_matr} stating that any two matrices in expansive normal form induce equivalent quasi-norms if and only if they are equal. Hence there is a natural one-to-one correspondence between the scales of anisotropic, homogeneous Besov spaces and matrices in expansive normal form. Moreover, 
Part (b) of Theorem \ref{thm:class_equiv_matr} provides an easily checked characterization of coarse equivalence. From these criteria one easily infers that there exist pairs $A,B$ of expansive matrices that induce the same scale of inhomogeneous spaces, but different scales of homogeneous spaces. 

We close the paper with the description of an algorithm to decide whether two expansive matrices $A$ and $B$ induce the same (homogeneous and/or inhomogeneous) Besov spaces.

\section{Anisotropic Besov spaces} 
\label{sect:ani_besov}

Our exposition regarding anisotropic Besov spaces  follows \cite{Bow05}. Let us start with some preliminaries and basic notions. We will use the following normalization of the Fourier transform $\mathcal{F} : {\rm L}^1(\mathbb{R}^d) \to C_0(\mathbb{R}^d)$: For all $f \in {\rm L}^1(\mathbb{R}^d)$,
\[
 \mathcal{F}(f)(\xi) = \widehat{f}(\xi) = \int_{\mathbb{R}^d} f(x) e^{- 2 \pi i \langle \xi, x \rangle} dx~.
\]
$\mathcal{S}(\mathbb{R}^d)$ denotes the space of Schwartz functions, $\mathcal{S}'(\mathbb{R}^d)$ its dual, the space of tempered distributions. As is well-known, the Fourier transform extends canonically to $\mathcal{S}'(\mathbb{R}^d)$. We let $\mathcal{P}$ denote the space of polynomials on $\mathbb{R}^d$, which can be viewed as a subspace of $\mathcal{S}(\mathbb{R}^d)$. For these definitions and basic properties of the Fourier transform, we refer to \cite{Ru_FA}. 

Given an open subset $\mathcal{O} \subset \mathbb{R}^d$, we let $\mathcal{D}(\mathcal{O}) = C_c^\infty(\mathcal{O})$, the space of smooth, compactly supported functions on $\mathcal{O}$, endowed with the usual topology \cite{Ru_FA}. We let $\mathcal{D}'(\mathcal{O})$ denote its dual space. 
We use ${\rm supp}(f) = \overline{f^{-1}(\mathbb{C} \setminus \{  0 \})}$ for the support of a function $f$. Given a Borel subset $C \subset \mathbb{R}^d$, $\lambda(C)$ denotes its Lebesgue measure. The cardinality of a set $X$ is denoted by $|X|$. 

Given a vector $x = (x_1,\ldots,x_d)^T \in \mathbb{R}^d$, we denote by $|x| = \left( \sum_{I=1}^d |x_i|^2 \right)^{1/2}$ its euclidean length. Given a matrix $A \in \mathbb{R}^d$, we let $\| A \| = \sup_{|x| = 1} |A x|$. 

The definition of anisotropic Besov spaces is based on the notion of expansive matrices. 
\begin{definition}
 \label{defn:expansive}
 A matrix $A \in {\rm GL}(d,\mathbb{R})$ is called {\bf expansive}, if all its (possibly complex) eigenvalues $\lambda$ fulfill $|\lambda|>1$. 
\end{definition}

\begin{definition}
 Let $A \in {\rm GL}(d,\mathbb{R})$ be an expansive matrix. $\psi \in \mathcal{S}(\mathbb{R}^d)$ is called {\bf $A$-wavelet} if it fulfills
 \begin{eqnarray}
  \label{eqn:def_wv1}  & & {\rm supp}(\widehat{\psi}) \subset [-1,1]^d \setminus \{ 0 \}~, \\
  \label{eqn:def_wv2}  & & \forall \xi \in \mathbb{R}^d \setminus \{0 \}~:~\sum_{j \in \mathbb{Z}} \left| \widehat{\psi}((A^T)^j \xi) \right|>0~.
 \end{eqnarray}
 Given a wavelet $\psi$, we define $\psi_j(x) = |{\rm det}(A)|^j \psi(A^jx)$, for $j \in \mathbb{Z}$.
 Given a wavelet $\psi$, a function $\psi_0 \in S(\mathbb{R})$ is called {\bf  low-pass complement to $\psi$}, if $\widehat{\psi}$ is compactly supported, with 
 \begin{equation}
  \forall  \xi \in \mathbb{R}^d ~:~ |\widehat{\psi}_0(\xi)| + \sum_{j \in \mathbb{N}_0} |\widehat{\psi}((A^T)^j\xi)| > 0~.
 \end{equation}
 The inhomogeneous wavelet system $(\psi_j^i)_{j \in \mathbb{N}_0}$ is defined by $\psi_j^i = \psi_j$, for $j \ge 1$, and $\psi_0^i = \psi_0$. 
\end{definition}

\begin{remark} \label{rem:wavelets}
 \begin{enumerate}
  \item[(a)] It is not hard to see that every expansive matrix $A$ allows the existence of $A$-wavelets. 
  \item[(b)] The fact that the Fourier transform $\widehat{\psi}$  of an $A$-wavelet has compact support away from zero actually implies 
 ${\rm supp}(\widehat{\psi}) \subset [-1,1]^d \setminus  \epsilon [-1,1]^d$ for some $\epsilon>0$. For the expansive matrix $A$, this implies that, for all $\xi \in \mathbb{R}^d$,
 \[
  |\{ j \in \mathbb{Z} : \widehat{\psi} ((A^T)^j  \xi) \not= 0 \}| \le M ~,
 \] with $M$ independent of $\xi$. As a consequence of this observation, the denominator of the right-hand side of 
 \[
  \widehat{\eta}(\xi) =  \frac{|\widehat{\psi}(\xi)|^2}{ \sum_{j \in \mathbb{Z}} \left| \widehat{\psi}((A^T)^j \xi) \right|^2}
 \] is locally finite, hence a smooth $C^\infty$-function, vanishing nowhere by Assumption (\ref{eqn:def_wv2}). 
 
 This establishes that $\widehat{\eta}$ a well-defined $C_c^\infty$-function, and $\eta \in \mathcal{S}(\mathbb{R}^d)$ is an $A$- wavelet satisfying the additional condition 
 \begin{equation}
  \label{eqn:def_wv2_strong} \forall \xi \in \mathbb{R}^d \setminus \{0 \}~:~\sum_{j \in \mathbb{Z}} \widehat{\eta}((A^T)^j \xi) = 1~.
 \end{equation}
 Similarly, one can construct $A$-wavelets $\tilde{\eta}$ satisfying the condition
 \begin{equation}
  \label{eqn:def_wv2_strong2} \forall \xi \in \mathbb{R}^d \setminus \{0 \}~:~\sum_{j \in \mathbb{Z}} \left| \widehat{\tilde{\eta}}((A^T)^j \xi) \right|^2= 1~.
 \end{equation}
 Hence we may replace assumption (\ref{eqn:def_wv2}) by (\ref{eqn:def_wv2_strong}) or by (\ref{eqn:def_wv2_strong2}), if it proves convenient.
 \item[(c)] Similarly, one can find Schwartz functions $\psi_0$ and $\psi$ such that
  \begin{equation}
  \label{eqn:def_wv2_strong_ih} \forall \xi \in \mathbb{R}^d~:~\widehat{\psi}_0(\xi) + \sum_{j \in \mathbb{N}} \widehat{\psi}((A^T)^j \xi) = 1~,
 \end{equation}
 or 
   \begin{equation}
  \label{eqn:def_wv2_strong2_ih} \forall \xi \in \mathbb{R}^d~:~|\widehat{\psi_0}(\xi)|^2 + \sum_{j \in \mathbb{N}} |\widehat{\psi}((A^T)^j \xi)|^2 = 1~,
 \end{equation}
holds. 
\item[(d)] The polynomials $p \in \mathcal{P} \subset \mathcal{S}'(\mathbb{R}^d)$ are characterized by the fact that ${\rm supp}(\widehat{p}) \subset \{ 0 \}$. Hence the convolution theorem yields for any $A$-wavelet $\psi$ that $(p \ast \psi_j)^\wedge = \widehat{p} \cdot \widehat{\psi_j} = 0$, and thus $p \ast \psi_j = 0$, for all $j \in \mathbb{Z}$.  
 \end{enumerate}
\end{remark}

\begin{definition}
 Let $A$ denote an expansive matrix, $\alpha \in \mathbb{R}$, and $0 < q \le \infty$. The sequence space $\ell^q_{v_{\alpha,A}}(\mathbb{Z})$ is the space of all sequences $(c_j)_{j \in \mathbb{Z}}$ with the property that $(|{\rm det}(A)|^{\alpha j} c_j)_{j \in \mathbb{Z}} \in \ell^q$, endowed with the obvious (quasi-)norm.  The space $\ell^q_{v_{\alpha,A}}(\mathbb{N}_0)$ is defined analogously. Since the precise meaning can usually be inferred from the context, we will typically write $\ell^q_{v_{\alpha,A}}$ for either of the two spaces.
\end{definition}

\begin{definition}
 \label{defn:an_bes}
 Let $\alpha \in \mathbb{R}$, $0 < p,q \le \infty$.  Let $A$ be an expansive matrix, and $\psi$ an $A$-wavelet, with low-pass complement $\psi_0$. 
 \begin{enumerate}
  \item[(a)]
 We define the {\bf anisotropic homogeneous Besov (quasi-) norm} by letting, for
 given $f \in \mathcal{S}'(\mathbb{R}^d)$,
 \begin{equation} \label{eqn:def_bnorm}
  \| f \|_{\dot{B}_{p,q}^\alpha(A)} = \left\| \left( \left\| f \ast \psi_j\right\| \right)_{j \in \mathbb{Z}} \right\|_{\ell^q_{v_{\alpha,A}}} 
 \end{equation} 
 We let  $\dot{B}_{p,q}^\alpha(A)$ denote the space of all tempered distributions $f$ with $\| f \|_{\dot{B}_{p,q}^\alpha(A)} < \infty$. We identify elements of $\dot{B}_{p,q}^\alpha(A)$ that only differ by a polynomial. 
 \item[(b)] The {\bf anisotropic inhomogeneous Besov (quasi-) norm} is defined for
 given $f \in \mathcal{S}'(\mathbb{R}^d)$ by 
 \begin{equation} \label{eqn:def_bnorm_ih}
  \| f \|_{{B}_{p,q}^\alpha(A)} = \left\| \left( \left\| f \ast \psi_j^i\right\| \right)_{j \in \mathbb{N}_0} \right\|_{\ell^q_{v_{\alpha,A}}} 
 \end{equation} 
 We let  $B_{p,q}^\alpha(A)$ denote the space of all tempered distributions $f$ with $\| f \|_{B_{p,q}^\alpha(A)} < \infty$. 
 \end{enumerate}
\end{definition}

\begin{remark} \label{rem:def_besov}
 A few words regarding well-definedness of the (quasi-)norm and the associated Besov space are in order. First of all, note that for any tempered distribution $f$, the convolution product $f \ast \psi_j$ is a smooth function of polynomial growth, that may or may not be $p$-integrable. By convention, the right-hand side (\ref{eqn:def_bnorm}) is infinite whenever one of the convolution products $f \ast \psi_j$ is not $p$-integrable. In case the sequence $\left( \left\| f \ast \psi_j\right\| \right)_{j \in \mathbb{Z}}$ consists only of finite real numbers, its $\ell^q_{\alpha,A}$-norm is declared infinite whenever the 
 sequence is {\em not} in $\ell^q_{\alpha,A}$.
 
 Note that strictly speaking, $\dot{B}_{p,q}^\alpha(A) \subset \mathcal{S}'(\mathbb{R}^d)/\mathcal{P}$.
 The well-definedness of the norm on the quotient space follows from the fact that $f \ast \psi_j = (f + p) \ast \psi_j$ for all $p \in \mathcal{P}$, by Remark \ref{rem:wavelets}(c). 
 
 With these conventions, the anisotropic Besov spaces and their norm are well-defined, although possibly dependent on the choice of wavelet. The independence of the norm (up to equivalence) of the choice of wavelet, and thus of the space, is shown in \cite[Corollary 3.7]{Bow05}. Furthermore, we mention the following properties of anisotropic Besov spaces: They are normed-spaces for $1 \le p,q \le \infty$, and quasi-normed otherwise. Furthermore, all spaces are {\em complete}, by \cite[Proposition 3.3]{Bow05}.
\end{remark}

\begin{remark} \label{rem:l2_besov}
Using a wavelet $\psi$  fulfilling the strong admissibility condtion (\ref{eqn:def_wv2_strong2}), one computes for $f \in \dot{B}_{2,2}^0(A)$ using a wavelet fulfilling the condition (\ref{eqn:def_wv2_strong})
\begin{eqnarray*}
 \| f \|_{\dot{B}_{2,2}^0}^2 & = & \sum_{j \in \mathbb{Z}} \| f \ast \psi_j \|_2^2 \\
 & = & \sum_{j \in \mathbb{Z}} \int_{\mathbb{R}^d} |\widehat{f}(\xi)|^2 |\widehat{\psi}(A^j \xi)|^2  d\xi \\
 & = & \int_{\mathbb{R}^d} |\widehat{f}(\xi)|^2 \underbrace{\sum_{j \in \mathbb{Z}}  |\widehat{\psi}(A^j \xi)|^2}_{\equiv 1} d\xi \\
 & = & \| f \|_2^2~.
\end{eqnarray*}
Here the second equality used the Plancherel theorem, as well as condition (\ref{eqn:def_wv2_strong}). A similar argument shows that $B_{2,2}^0(A) = {\rm L}^2(\mathbb{R}^d)$. 
\end{remark}

The chief aim of this paper is to understand the dependence of the scale of anisotropic Besov spaces on the underlying matrix. For this reason, we define an equivalence relation:
\begin{definition}
 Let $A$ and $B$ denote expansive matrices. We write $A \sim_{\dot{B}} B$ whenever $\dot{B}_{p,q}^\alpha (A) =  \dot{B}_{p,q}^\alpha (B)$ holds for all $0 < p,q \le \infty, \alpha \in \mathbb{R}$.
 The relation $A \sim_B B$ is defined analogously. 
\end{definition}
%


\section{Decomposition spaces}
\label{sect:dec_spaces}

Decomposition spaces were introduced by Feichtinger and Gr\"obner \cite{DecompositionSpaces1,DecompositionSpaces2}, initially with the aim of constructing intermediate spaces between (isotropic) Besov spaces and modulation spaces. The decomposition space formalism is an extremely flexible tool for the description of elements of large variety of function spaces in terms of their Fourier localization,  including $\alpha$-modulation spaces, Besov spaces, curvelet smoothness spaces, wavelet coorbit spaces over general dilation groups, etc.; see \cite{DecompositionSpaces1,DecompositionSpaces2,BorupNielsenDecomposition,BorupNielsenAlphaModulationSpaces,FuVo,VoDiss}. In order to treat homogeneous Besov spaces along with the inhomogeneous ones, the initial definition had to be somewhat modified, to allow for decompositions that do not cover the full frequency space \cite{FuVo}. 

As will become clear below, our paper further contributes to this unifying view onto function spaces through the lense of decomposition space theory. Let us now start by recounting the definition of Fourier-side decomposition spaces, in the form defined in \cite{VoDiss}.
These spaces depend on certain coverings of a suitably chosen set $\mathcal{O} \subset \mathbb{R}^d$ of frequencies, and partitions of unity subordinate to these. For the purposes of this paper, it is sufficient to treat a particularly amenable class of coverings, described in the next definition. 

\begin{definition}
 Let $\mathcal{O} \subset \mathbb{R}^d$ be open, and let $\mathcal{Q} = (Q_i)_{i \in I}$ denote a family of subsets $Q_i \subset \mathcal{O}$ with compact closure in $\mathcal{O}$.
 \begin{enumerate}
  \item[(a)] We call $\mathcal{Q}$ an {\bf admissible covering} of $\mathcal{O}$, if it fulfills the following conditions: 
 \begin{enumerate}
   \item[(i)] {\bf Covering property:} $\mathcal{O} = \bigcup_{i \in I} Q_i$
  \item[(ii)] {\bf Admissibility:} $\sup_{i \in I} \sup_{j \in I, Q_i \cap Q_j \not= \emptyset} \frac{\lambda(Q_i)}{\lambda(Q_j)} <\infty$.
 \end{enumerate}
 \item[(b)] $\mathcal{Q}$ is called an {\bf almost structured admissible covering} if it is an admissible covering, and there exists a family $(Q_i')_{i \in I}$ of open bounded sets as well as $T_i \in {\rm GL}(d,\mathbb{R})$ and $b_i \in \mathbb{R}^d$ fulfilling the following conditions:
 \begin{enumerate}
  \item[(i)] For all $i \in I$: $\overline{T_i Q_i' + b_i} \subset Q_i$.
  \item[(ii)] The quantity $\sup_{i,j: Q_i \cap Q_j \not= \emptyset} \| T_i^{-1} T_j \|$  is finite.
  \item[(iii)] The set $\{ Q_i': i \in I \}$ is finite. 
  \item[(iv)] The family $(T_i Q_i' + b_i)_{ i \in I}$ is an admissible covering. 
 \end{enumerate}
 The tuple $((T_i)_{i \in I}, (b_i)_{i \in I}, (Q_i')_{i \in I})$ are called {\bf standardization} of $\mathcal{Q}$. 
\end{enumerate}
\end{definition}

The subtleties connected to the various notions of coverings is the price one has to pay for the generality of the decomposition space approach.

The definition of associated function spaces uses a particular class of partitions of unity subordinate to an admissible covering. 
\begin{definition}
 Let $\mathcal{Q} = (Q_i)_{i \in I}$ denote an almost structured admissible covering with standardization  $((T_i)_{i \in I}, (b_i)_{i \in I}, (Q_i')_{i \in I})$, and let $0 < p \le \infty$. We call a family $(\varphi_i)_{i \in I}$ of functions an {\bf ${\rm L}^p$-BAPU} with respect to $\mathcal{Q}$ if it has the following properties:
 \begin{enumerate}
 \item[(i)] For all $i \in I$~:~$\varphi_i \in C_c^\infty( \mathcal{O} )$.
 \item[(ii)] For all $i \in I$~:~$\varphi_i \equiv 0$ on $\mathbb{R}^d \setminus Q_i$.
 \item[(iii)] $\sum_{i \in I} \varphi \equiv 1$ on $\mathcal{O}$. 
 \item[(iv)] $\sup_{i \in I} |{\rm det}(T_i)|^{\frac{1}{t}-1} \| \mathcal{F}^{-1} \varphi_i \|_{{\rm L}^p} <\infty$.
 Here $t = \min(p,1)$. 
 \end{enumerate}
\end{definition}
The word BAPU in the definition is an acronym for {\em bounded amissible partition of unity}. 
The following important remark ensures that BAPUs exist for almost structured admissible coverings \cite[Theorem 2.8]{VoDiss}.
\begin{lemma}
 Let $\mathcal{Q} = (Q_i)_{i \in I}$ denote an almost structured admissible covering. Then there exists a family $(\varphi_i)_{i \in I}$ that is an ${\rm L}^p$-BAPU, for every $0 < p \le \infty$. 
\end{lemma}

\begin{definition}
 Let $\mathcal{Q} = (Q_i)_{i \in I}$ denote an admissible covering, and let $v: I \to \mathbb{R}^+$ denote a weight. The weight is called {\bf $\mathcal{Q}$-moderate} if 
 \[
  \sup_{i,j \in I : Q_i \cap Q_j \not= \emptyset} \frac{v(i)}{v(j)} < \infty~. 
 \]
Given $0 < q \le \infty$, we define $\ell^q_v(I) = \{ c = (c_i)_{i \in I} \in \mathbb{C}^I~:~ (c_i v(i))_{i \in I} \in \ell^q(I) \}$, endowed with the obvious (quasi-)norm. 
\end{definition}

We can now define the class of decomposition spaces that we are interested in:
\begin{definition}
 Let $\mathcal{Q} = (Q_i)_{i \in I}$ denote an almost structured admissible covering, let $v$ denote a $\mathcal{Q}$-moderate weight on $I$, and $0 \le p,q \le \infty$. Let $(\varphi_i)_{i \in I}$ denote an ${\rm L}^p$-BAPU associated to $\mathcal{Q}$. Given $u \in \mathcal{D}'(\mathcal{O})$, we define its {\bf decomposition space (quasi-)norm} as 
 \begin{equation}
  \label{eqn:def_dsnorm} \left\| u \right\|_{\mathcal{D}(\mathcal{Q},{\rm L}^p,\ell^q_v)} = \left\| \left( \| \mathcal{F}^{-1} (\varphi_i \cdot u ) \|_{{\rm L}^p} \right)_{ i \in I} \right\|_{\ell^q_v}~.
 \end{equation}
 We denote by $\mathcal{D}(\mathcal{Q},,{\rm L}^p,\ell^q_v)$ the space of all $u \in \mathcal{D}'(\mathcal{O})$ for which this (quasi-)norm is finite.
\end{definition}

\begin{remark}
\begin{enumerate}
 \item[(a)]
 We use the same conventions regarding finiteness of the (quasi-) norm as for the Besov space setting, see Remark \ref{rem:def_besov}.  For well-definedness of the inverse Fourier transform  $\mathcal{F}^{-1} (\varphi_i \cdot u ) $ observe that the pointwise product $\varphi_i \cdot u$ is a distribution on $\mathbb{R}^d$ with compact support, hence is a {\em tempered} distribution, whose inverse Fourier transform is a smooth function. Hence the meaning of the ${\rm L}^p$-norm of the inverse Fourier transform is clear, if one allows $\infty$ as a possible value. 

Just as for Besov spaces, the (quasi-)norms of decomposition spaces are independent of the choice of BAPU (up to equivalence), and complete with respect to their (quasi-)norms. Also, an application of the Plancherel theorem, similar to that in the Besov space case, easily establishes that $\mathcal{D}(\mathcal{Q},L^2,\ell^2) = {\rm L}^2(\mathbb{R}^d)$, whenever $\mathbb{R}^d \setminus \bigcup_{i \in I}  Q_i$ has measure zero. 
\item[(b)] It is common in the decomposition space literature to use the space $\mathcal{S}'(\mathbb{R}^d)$ as a reservoir space for $\mathcal{D}(\mathcal{Q},{\rm L}^p,\ell_v^q)$, rather than the larger space $\mathcal{D}'(\mathcal{O})$. This may result in {\em incomplete} decomposition spaces, see \cite[Remark after Definition 21]{FuVo}, and it is the main reason why our definition relies on $\mathcal{D}'(\mathcal{O})$. This distinction will become relevant in the proof of Theorem \ref{thm:besov_as_decsp} below. 
\end{enumerate}

 \end{remark}

The definition of decomposition spaces puts the frequency covering at the center of attention: Important properties of decomposition spaces should be related to properties of the underlying covering, and a major task of decomposition space theory is to make this relation explicit and transparent. The thesis \cite{VoDiss} demonstrates that this programme can be carried out for {\em embedding theorems} describing inclusion relations between decomposition spaces, see also \cite{Vo_Embed1}. In the following, we will be concerned with a much more restrictive question, namely that of {\em equality} of decomposition spaces: We would like to understand when different coverings induce the same decomposition spaces. For this purpose, we must be able to compare different admissible coverings. The pertinent notions for such a comparison are contained in the following definitions.

\begin{definition}
 Let $\mathcal{Q} = (Q_i)_{i \in I}$ and $\mathcal{P} = (P_j)_{j \in J}$ denote admissible coverings of the open set $\mathcal{O}$. 
 \begin{enumerate}
  \item[(a)] Given $i \in I$ and $n \in \mathbb{N}_0$, we inductively define index sets $i^{n*} \subset I$ via
  \[
   i^{0*} = \{ i \}~,~i^{(n+1)*} = \{ j \in I~:~\exists k \in i^{n*} \mbox{ with } Q_k \cap Q_j \not= \emptyset \}~. 
  \]
  \item[(b)] We define 
  \[
   Q_{i}^{n*} = \bigcup_{j \in i^{n*}} Q_j~.
  \]
  \item[(c)] We call $\mathcal{Q}$ {\bf  almost subordinate to } $\mathcal{P}$ if there exists $k \in \mathbb{N}_0$ such that for every $i \in I$ there exists a $j_i \in J$ with 
  $Q_i \subset P_{j_i}^{k*}$. 
  \item[(d)] $\mathcal{Q}$ and $\mathcal{P}$ are called {\bf equivalent} if $\mathcal{Q}$ is almost subordinate to $\mathcal{P}$ and $\mathcal{P}$ is almost subordinate to $\mathcal{Q}$. 
 \end{enumerate}
\end{definition}

\begin{definition}
 Let $\mathcal{Q}=(Q_i)_{i \in I}$ and $\mathcal{P}= (P_j)_{j \in J}$ denote admissible coverings of the open sets $\mathcal{O}$ and $\mathcal{O}'$, respectively. 
 \begin{enumerate}
  \item[(a)] Given $i \in I$, we define $J_i \subset J$ as
  \[
   J_i = \{ j \in J: P_j \cap Q_i \not= \emptyset \}~,
  \]
  and similarly, for $j \in J$, 
  \[
   I_j =   \{ i \in I: P_j \cap Q_i \not= \emptyset \}~.
  \]
 \item[(b)] $\mathcal{Q}$ and $\mathcal{P}$ are called {\bf weakly equivalent} if
 \[
  \sup_{i \in I} | J_i | + \sup_{j \in J} |I_j|  < \infty 
 \]
 \end{enumerate}
\end{definition}
%
%

It is useful to also have a notion of equivalence for weights over different coverings. 
\begin{definition} \label{defn:weight_equiv}
 Let $\mathcal{Q} = (Q_i)_{i \in I}$ and $\mathcal{P} = (P_j)_{j \in J}$ denote admissible coverings,  $v_1$ a $\mathcal{Q}$-moderate weight on $I$, and $v_2$ a $\mathcal{P}$-moderate weight on $J$. We define
 \[
  v_1 \asymp v_2 :\Leftrightarrow   \sup_{i \in I, j \in J~:Q_i \cap P_j \not= \emptyset} \frac{v_1(i)}{v_2(j)} +  \frac{v_2(j)}{v_1(i)} < \infty
 \]
\end{definition}

We now cite two results relating (weak) equivalence of coverings to equality of decomposition spaces. The first one can be understood as a rigidity theorem, 
 see \cite[Theorem 1.10]{Vo_Embed1}.
\begin{theorem} \label{thm:rigidity}
 Let $\mathcal{Q}=(Q_i)_{i \in I}$ and $\mathcal{P}= (P_j)_{j \in J}$ denote almost structured admissible coverings of the same open set $\mathcal{O}$, 
  $v_1$ a $\mathcal{Q}$-moderate weight on $I$, and $v_2$ a $\mathcal{P}$-moderate weight on $J$. Assume that  $(p_1,q_1,p_2,q_2) \in (0,\infty]^4 \setminus \{ (2,2,2,2) \}$ exist with 
 \[
   \mathcal{D}(\mathcal{Q},{\rm L}^{p_1}, \ell^{q_1}_{v_1}) = \mathcal{D}(\mathcal{P},{\rm L}^{p_2}, \ell^{q_2}_{v_2})~.
 \]
 Then $(p_1,q_1) = (p_2,q_2)$, $v_1 \asymp v_2$, and $\mathcal{P}$ and $\mathcal{Q}$ are weakly equivalent. 
\end{theorem}

The converse requires somewhat different conditions. The following generalizes results from \cite{DecompositionSpaces1,BorupNielsenDecomposition}. 
\begin{lemma} \label{lem:suf_dc_equal}
  Let $\mathcal{Q}=(Q_i)_{i \in I}$ and $\mathcal{P}= (P_j)_{j \in J}$ denote weakly equivalent almost structured admissible coverings of the same open set $\mathcal{O}$, with standardizations $((T_i)_{i \in I}, (b_i)_{i \in I}, (Q_i')_{i \in I})$ and $((S_j)_{j \in J}, (c_j)_{j \in J},(P_j')_{j \in J})$, respectively.
  Let $v_1$ denote a $\mathcal{Q}$-moderate weight on $I$, and $v_2$ a $\mathcal{P}$-moderate weight on $J$, with $v_1\asymp v_2$.
  Assume finally that $\mathcal{Q}$ is almost subordinate to $\mathcal{P}$, and that there exists a $C>0$ such that 
  \[ \forall (i,j) \in I \times J:  \left(Q_i \cap P_j \not= \emptyset \Rightarrow \| T_i^{-1} S_j \| + \| S_j^{-1} T_i \| \le C \right) \]
 Then 
\[
 \mathcal{D}(\mathcal{Q},{\rm L}^{p}, \ell^{q}_{v_1}) = \mathcal{D}(\mathcal{Q},{\rm L}^{p}, \ell^{q}_{v_2})
\] holds for all $0 < p,q \le \infty$.
\end{lemma}
\begin{proof}
 This is \cite[Lemma 6.10]{Vo_Embed1}. Note that this result also requires an upper bound on $ |{\rm det}(S_j^{-1} T_i)|$, for all $i,j$ with $Q_i \cap P_j \not= \emptyset$, which in our setting follows from the bound on the norms.
\end{proof}

\section{Expansive Matrices and Homogeneous Quasi-Norms}

\label{sect:exp_matr}

In this section we collect the pertinent properties of expansive matrices. For more background on the following definitions and results, we refer to \cite{Bow03}.
Throughout this section, let $A$ denote an expansive matrix. We first note a number of observations regarding norms of (powers of) expansive matrices:

\begin{lemma} \label{lem:exp_norm}
 Let $\lambda_1,\ldots,\lambda_d \in \mathbb{C}$ denote the eigenvalues of $A$, counted with their algebraic multiplicities, and numbered to ensure
 \[
  1 < |\lambda_1| \le |\lambda_2| \le \ldots \le |\lambda_d| ~.
 \] Pick $1 < \lambda_-< |\lambda_1|$ and $\lambda_+ > |\lambda_d|$. Then there exists a constant $c>0$ such that, for all $j \ge 0$ and $x \in \mathbb{R}^d$:
 \begin{eqnarray*}
  \frac{1}{c} \lambda_-^j |x| & \le &  |A^j x| \le c \lambda_+^j |x| \\
  \frac{1}{c} \lambda_+^{-j} |x| & \le &  |A^{-j} x| \le c \lambda_-^{-j} |x| \\ 
 \end{eqnarray*}
As a consequence, we have 
\[
 \frac{1}{c} \lambda_-^j  \le  \| A^j \| \le c \lambda_+^j ~,~ \frac{1}{c} \lambda_+^{-j} \le  \|A^{-j} \| \le c \lambda_-^{-j} ~.
\]
\end{lemma}

The following norm estimate will be useful, see \cite[Lemma 10.1]{Bow03}
\begin{lemma} \label{lem:norm_est}
 Let $c_1,c_2 >0$. Then there is $c_3 = c_3(c_1,c_2)>0$ such that for all matrices $A$ with $\| A \| \le c_1$ and $|{\rm det}(A)| \ge c_2$, the estimate $\| A^{-1} \| \le c_3$ holds.
\end{lemma}
\begin{proof}
 By Cramer's rule, the entries of $A^{-1}$ are polynomials in $|{\rm det}(A)|^{-1}$ and the entries of $A$, and our assumptions provide upper bounds for both.
\end{proof}

\begin{definition} 
 An $A$-homogeneous quasi-norm is a Borel map  $\rho_A : \mathbb{R}^d \to \mathbb{R}_0^+$ satisfying the following conditions:
 \begin{enumerate}
  \item[(i)] $\rho_A(x) = 0$ if and only if $x = 0$.
  \item[(ii)] {\bf $A$-homogeneity}: $\rho_A(Ax) = |{\rm det}(A)| \rho_A(x)$.
  \item[(iii)] {\bf Triangle inequality}: There exists a constant $C>0$ such that, for all $x,y \in \mathbb{R}^d$,
  \[
   \rho_A(x+y) \le C(\rho_A(x)+\rho_A(y))~.
  \]
 \end{enumerate}
\end{definition}

The next lemma yields that any two quasi-norms that are homogeneous with respect to the same expansive matrix $A$ are equivalent. In the following, we use the term {\bf ellipsoid} for images $CB_1(0)$ of the unit ball (with respect to the euclidean norm) under some invertible matrix $C$. 
\begin{lemma}
\begin{enumerate}
 \item[(i)] Any two $A$-homogeneous norms on $\mathbb{R}^d$ are equivalent, i.e., given two such mappings $\rho_1,\rho_2$, there exists a constant $C \ge 1$ such that, for all $x \in \mathbb{R}^d$:
 \[
  \frac{1}{C} \rho_1(x) \le \rho_2(x) \le C \rho_1(x)~.
 \]
 \item[(ii)] There exists an ellipsoid $\Delta_A$ and $r>1$ such that 
 \[
  \Delta_A \subset r \Delta_A \subset A \Delta_A~,
 \] and $\lambda(\Delta_A) = 1$. Then, letting 
 \[
  \rho_A(x) = |{\rm det}(A)|^j 
 \] for $x \in A^{j+1} \Delta_A \setminus A^j \Delta_A$, and $\rho(0) = 0$, defines an $A$-homogeneous quasi-norm. 
\end{enumerate}
\end{lemma}

As will be seen below, equivalence of induced quasi-norms is closely related to the equivalence relation $\sim_{\dot{B}}$. By contrast, the equivalence relation induced by the {\em inhomogeneous} spaces will be seen to depend on a slightly less restrictive type of equivalence:
\begin{definition} \label{defn:coarse}
 Let $\rho_A$ and $\rho_B$ denote two quasi-norms on $\mathbb{R}^d$. $\rho_A$ and $\rho_B$ are called {\bf coarsely equivalent} if there exist constants $c \ge 1$ and $R \ge 0$ such that
 \[
  \frac{1}{c} \rho_A - R \le \rho_B \le c \rho_A + R~. 
 \]
We call two expansive matrices $A$ and $B$ (coarsely) equivalent if and only if the induced quasi-norms are (coarsely) equivalent. 
\end{definition}

\begin{remark} \label{rem:coarse_equiv}
 \begin{enumerate}
  \item[(a)] The notion of coarse equivalence originates from geometric group theory, and also plays a role in operator theory and global analysis; see \cite{Roe} for an introduction. Clearly, equivalent quasi-norms are also coarsely equivalent. The converse will be seen to be false in general.
  \item[(b)] $\rho_A$ and $\rho_B$ are coarsely equivalent if and only if there exists $R>0$ and $c \ge 1$ with the property that for all $x \in \mathbb{R}$ with $|x| \ge R$, the inequalities
  \[
   \frac{1}{c} \rho_A(x) \le \rho_B(x) \le c \rho_A(x) ~.
  \]
  I.e., coarse equivalence can be understood as equivalence of the quasi-norms ``at infinity''. 
The elementary proof of the equivalence uses that $\rho_A$ and $\rho_B$ are bounded on compact sets. 
 \end{enumerate}
\end{remark}

We will be interested in understanding when two different matrices are (coarsely) equivalent. As a first step in this direction, we want to translate (coarse) equivalence of $A$ and $B$ to conditions involving certain products of powers of the two matrices. We first introduce a quantity that will frequently appear in the following criteria.
\begin{definition} \label{defn:eps}
 Let $A$ and $B$ be two expansive matrices. We let
 \[
  \epsilon(A,B) = \frac{\ln (|{\rm det}(A)|)}{\ln (|{\rm det}(B)|)}~.
 \]
\end{definition}

 Every real-valued matrix $A$ can be understood as inducing a linear map on $\mathbb{C}^d$, and in the following, eigenvalues and corresponding eigenvectors of such a matrix $A$ will be understood as possibly complex-valued. 
 See \cite[Lemma (10.2)]{Bow03} for a proof of the following statement. 
\begin{lemma} \label{lem:norm_equiv}
Let $A$ and $B$ be two expansive matrices. Then $A$ and $B$ are equivalent if and only if
$\sup_{k \in \mathbb{Z}} \left\| A^{-k} B^{\lfloor \epsilon k \rfloor} \right\| < \infty$  holds,
 with $\epsilon = \epsilon(A,B)$.
\end{lemma}

\begin{remark} \label{rem:transpose}
 As a consequence of the characterization, we note that $A^T$ and $B^T$ are equivalent if and only if $A$ and $B$ are. 
 The first statement is equivalent to 
 \begin{eqnarray*}
  \infty & > & \sup_{k \in \mathbb{Z}} \left\| (A^T) ^{-k} (B^{T})^{\lfloor \epsilon k \rfloor} \right\|
  \\ &  = & \sup_{k \in \mathbb{Z}} \left\| \left( B^{\lfloor \epsilon k \rfloor} A^{-k} \right)^T \right\| \\
  & = & \sup_{k \in \mathbb{Z}} \left\|  B^{\lfloor \epsilon k \rfloor} A^{-k} \right\|
 \end{eqnarray*}
 whereas condition (c), applied to $A$ and $B$ in reverse order, yields
 \[
  \sup_{\ell \in \mathbb{Z}} \left\| B^{-\ell} A^{\lfloor \ell/\epsilon \rfloor} \right\| < \infty\, ,
 \] and it is not hard to see that these two conditions are equivalent.
 
 We note that the analogous statement for coarse equivalence is wrong, see the remark following Theorem \ref{thm:class_equiv_matr}. 
\end{remark}

We next give a version of Lemma \ref{lem:norm_equiv} for coarse equivalence, which will be central to the following. This is why its proof, which is a straightforward adaptation of the proof of  \cite[Lemma (10.2)]{Bow03}, is included.
\begin{lemma} \label{lem:char_cequiv_matr1} Let $A$ and $B$ be two expansive matrices.
 Then $A$ and $B$ are coarsely equivalent if and only if 
 \begin{equation} \label{eqn:cequiv_norm}
  \sup_{k \in \mathbb{N}} \| A^{-k} B^{\lfloor \epsilon k \rfloor} \| < \infty~,
 \end{equation} with $\epsilon = \epsilon(A,B)$.
\end{lemma}
\begin{proof}
Throughout the proof, let $\epsilon = \epsilon(A,B)$.  First assume that (\ref{eqn:cequiv_norm}) holds. We note that it is equivalent to 
 \begin{equation} \label{eqn:cequiv_norm_v}
 \sup_{k \ge k_0} \| A^{-k} B^{\lfloor \epsilon k \rfloor} \| < \infty
 \end{equation}
for any fixed $k_0 \in \mathbb{Z}$. Furthermore, one has for all $k \in \mathbb{Z}$ that 
\begin{equation} \label{eqn:dets_equiv}
1 \le  |\det(A^{-k} B^{\lfloor \epsilon k \rfloor})| \le |\det(B)|~,
\end{equation} Hence, Lemma \ref{lem:norm_est} and (\ref{eqn:cequiv_norm_v}) also imply 
\[
  \sup_{k \ge k_0} \| B^{-\lfloor \epsilon k \rfloor} A^k \| < \infty
\]
Hence there exist constants $0 < C \le D < \infty$, depending only on $k_0 \in \mathbb{Z}$  such that 
\[
 \forall k \ge k_0 \,, \, \forall z \in \mathbb{R}^d \setminus \{ 0 \} ~: ~ C \le \frac{|A^{-k} z|}{|B^{-\lfloor \epsilon k\rfloor} z|} \le D ~. 
\]

Fix $r>0$ such that for all $x \in \mathbb{R}^{d} \setminus \{ 0 \}$ there exists $k \in \mathbb{Z}$ such that
\[
 1 \le |A^{-k} x | \le r~,~ 1 \le |B^{-k} x | \le r~,
\] e.g., $r = \max(\| A \|,\| B \|)$.  Fix $x \in \mathbb{R}^d$ with $|x| \ge 1$, and let $k \in \mathbb{Z}$ with $1 \le |A^{-k} x| \le r$.   

By Lemma \ref{lem:exp_norm}, we have for a suitable number $\lambda_+>0$ depending on $A$ 
\[
 r \ge |A^{-k} x| \ge |x| \| A^{k}\|^{-1} \ge \lambda_+^{-k} /c
\] hence $k \ge k_0$,  with $k_0$ only depending on $A$. 

Define
\begin{eqnarray*}
 c_A = \inf \{ \rho_A(z) : 1 \le |z| \le r \}, & & d_A = \sup \{ \rho_A(z) : 1 \le |z| \le r \}, \\
 c_B = \inf \{ \rho_B(z) :  1/D \le |z| \le r/C \}, & & d_B = \sup \{ \rho_B(z) : 1/D \le |z| \le r/C \}~.
\end{eqnarray*}
Now the choice of $k$ and homogeneity of $\rho_A$ yields 
\begin{equation} \label{eqn:ineq_rhoA}
c_A |{\rm det}(A)|^k \le \rho_A(x) \le d_A |{\rm det}(A)|^k ~.
\end{equation} Furthermore, the choice of $k$ yields via (\ref{eqn:cequiv_norm_v}) that 
\begin{equation}\label{eqn:ineq_rhoB}
 1/D  \le |B^{-\lfloor \epsilon k \rfloor} x| \le r/C~,
\end{equation}
whence we get via 
\begin{eqnarray*}
 \rho_B(x) & = & |{\rm det}(B)|^{\lfloor \epsilon k \rfloor} \rho_B(B^{-\lfloor \epsilon k \rfloor} x) \le |{\rm det}(B)|^{\lfloor \epsilon k \rfloor} d_B \\
 & \le & |{\rm det}(A)|^k |{\rm det}(B)| d_B \le \rho_A(x) |\det(B)| d_B/c_A~.
\end{eqnarray*}
By a similar argument, 
\[
 \rho_B(x) \ge \rho_A(x) c_B/d_A~.
\]
To summarize, we have found constants $C_1,C_2>0$ such that, for all $x$ with $|x|\ge 1$, $C_1 \rho_A(x) \le \rho_B(x) \le C_2 \rho_A(x)$. Hence the two quasi-norms are coarsely equivalent, by Remark \ref{rem:coarse_equiv}(b).

For the converse statement, assume that $\rho_A$ and $\rho_B$ are equivalent, i.e., for suitable $R,C> 0$ and all $x \in \mathbb{R}$ with $|x|\ge R$, the inequalities
$1/C \rho_B(x) \le \rho_A(x) \le C \rho_B(x)$ hold.

By Lemma \ref{lem:norm_est}, there exists $\ell_0 \in \mathbb{N}$ such that $|B^{\ell}(x)| \ge |x|$ holds for all $\ell \ge \ell_0$. For all $k_0 \ge \lceil \ell_0/\epsilon \rceil +1$ and all
$x \in \mathbb{R}^d$, with $|x| = R$, it follows that $|B^{\lfloor \epsilon k \rfloor} x| \ge |x| \ge R$, hence the coarse equivalence assumption gives rise to 
\begin{eqnarray*}
 \rho_A (A^{-k} B^{\lfloor \epsilon k \rfloor}x) & = & |\det(A)|^{-k} \rho_A(B^{\lfloor \epsilon k \rfloor} x) =  C |\det(A)|^{-k} \rho_B(B^{\lfloor \epsilon k \rfloor} x) \\
 & = & C |\det(A)|^{-k} |\det(B)|^{\lfloor \epsilon k \rfloor} \rho_B(x) \le \underbrace{C |\det{B}| \sup \{ \rho_B(x) : |x| = R \}}_{=K}~. 
\end{eqnarray*}
Thus we have that 
\[
 \{ A^{-k} B^{\lfloor \epsilon k \rfloor} x : k \ge k_0, |x| = R \} \subset \{ x \in \mathbb{R}^d : \rho_A(x) \le K \} ~, 
\] and the right-hand side is bounded, hence contained in a ball of radius $R_0$ with respect to the euclidean norm. But this implies
\[
 \sup_{k \ge k_0} \| A^{-k} B^{\lfloor \epsilon k \rfloor} \| \le R_0~,
\] and the converse is shown. 
\end{proof}

The following remark notes some elementary consequences of Lemmas \ref{lem:norm_equiv} and \ref{lem:char_cequiv_matr1}. 
\begin{remark} \label{rem:equiv_JNF}
\begin{enumerate}
\item[(a)] Assume that the matrices $A,B$ have the same block diagonal structure
\[  A  = \left( \begin{array}{cccc} A_1 & & &    \\  & A_2 & & \\ & & \ddots & \\ & & & A_k \end{array} \right) \,\, , \,\,  B  = \left( \begin{array}{cccc} B_1 & & &    \\  & B_2 & & \\ & & \ddots & \\ & & & B_k \end{array} \right)
\] with the additional property that $\epsilon(A_i,B_i) = \epsilon(A,B)$, for $i=1,\ldots,k$. Then a straightforward application of the criteria for (coarse) equivalnce yields that $A$ and $B$ are (coarsely) equivalent if and only if $A_i$ and $B_i$ are, for all $i=1,\ldots,k$. 
\item[(b)] If $A,B$ and $A',B'$ are related by $A  = CA'C^{-1}$ and $B=CB'C^{-1}$, then $A$ and $B$ are (coarsely) equivalent if and only if $A'$ and $B'$ are. 
\end{enumerate}
\end{remark}

\section{Anisotropic Besov spaces viewed as decomposition spaces}

\label{sect:ani_dec}

We will now establish the connection between  anisotropic Besov spaces and decomposition spaces. First we need to introduce a class of coverings.

\begin{definition}
 Let $A$ denote an expansive matrix. Let $C \subset \mathbb{R}^d$ be open, such that $\overline{C}$ is a compact subset of $\mathbb{R}^d \setminus \{ 0 \}$, and define, for $j \in \mathbb{Z}$, 
 \[
  Q_j = A^j \overline{C}~.
 \]
 If $\bigcup_{j \in \mathbb{Z}} Q_j = \mathbb{R}^d \setminus \{ 0 \}$, 
 $\mathcal{Q} = (Q_j)_{j \in \mathbb{Z}}$ is called {\bf homogeneous covering induced by $A$}. An {\bf inhomogeneous covering induced by $A$} is given by the family $\mathcal{Q}_A^i = (Q_j^i)_{j \in \mathbb{N}_0}$, where $Q_j^i = Q_j = A^j \overline{C}$ for $j \ge 1$, and $Q_0^i = \overline{C_0}$, for a relatively compact open set $C_0$ with the property that
 \[
  \bigcup_{j \in \mathbb{N}_0} Q_j^i = \mathbb{R}^d~. 
 \]
 \end{definition}

 \begin{lemma} \label{lem:ind_cv_adm}
  Let $A$ denote an expansive matrix, $\mathcal{Q} = (A^j Q_0)_{j \in \mathbb{Z}}$ a covering induced by $A$, and $\mathcal{Q}^i$ an associated inhomogeneous covering. Then $\mathcal{Q}, \mathcal{Q}^i$ are almost structured admissible coverings.
 \end{lemma}
\begin{proof}
First let us consider the homogeneous case.
It is clear that the homogeneous covering is almost structured, with standardization provided by $T_j = A^j$, $Q_j' = Q_0$ and $b_j  = 0$, for $j \in \mathbb{Z}$. Hence the only missing property is admissibility. By assumption, $Q_0 \subset \mathbb{R}^d$ has compact closure in $\mathbb{R}^d \setminus \{ 0 \}$.
Defining, for $R>1$, the {\bf annulus} 
 \[
  C_R = \{ x \in \mathbb{R}^d : R^{-1} < |x| < R \}~,
 \] for $R>1$. Compactness of $Q_0$ implies $Q_0 \subset C_R$ for $R$ sufficiently large. Now Lemma \ref{lem:exp_norm} implies the existence of $j_0 \in \mathbb{N}$ such that 
 \[ Q_j \cap Q_0 \subset A^j C_R \cap C_R = \emptyset \]
 for all $j$ with $|j|>j_0$. It follows that $Q_j \cap Q_i = \emptyset$ whenever $|j-i|>j_0$, and that entails admissibility of the covering. 

 In the inhomogeneous case, we use the same standardization as for the homogeneous case for positive $j$, as well as $Q_0 = C_0$, $T_0 = I_d$ and $b_0$, and see that the covering is almost structured. The remainder of the argument follows as before.
 \end{proof}

 \begin{lemma}
  Let $\mathcal{Q}$ and  $\mathcal{P}$ denote two coverings induced by the same matrix $A$, either both homogeneous or both inhomogeneous. Then $\mathcal{Q}$ and $\mathcal{P}$ are equivalent. 
 \end{lemma}
\begin{proof}
 We first treat the homogeneous case. By assumption, $Q_j = A^j Q_0$ and $P_i = A^i P_0$.  Since $\mathbb{R}^d \setminus \{ 0 \}$ is connected, it follows that 
 \[
  \bigcup_{k \in \mathbb{N}} P_0^{k*} = \mathbb{R}^d \setminus \{ 0 \}~,
 \] and since the interiors of the $P_0^{k*}$ are open and cover $\mathbb{R}^d \setminus \{ 0\}$, and $Q_0$ is relatively compact in $\mathbb{R}^d \setminus \{ 0 \}$, it follows that
 \[
  Q_0 \subset P_0^{k*}
 \] for some $k \in \mathbb{N}$. But then, by construction of the induced coverings,
 \[
  Q_j  = A^j Q_0 \subset A^j P_0^{k*} = P_j^{k*}~,
 \] for all $j \in \mathbb{Z}$, and thus $\mathcal{Q}$ is almost subordinate to $\mathcal{P}$. Now symmetry yields equivalence of the coverings. 

 Now $Q_j \cap P_i \not= \emptyset$ is equivalent to 
 \[
 A^{j-i} Q_0 \cap P_0 \not= \emptyset~.
 \] Using $Q_0,P_0 \subset C_R$ for $R$ sufficiently large, followed by the argument from the proof of Lemma \ref{lem:ind_cv_adm}, 
 yields that $|j-i|<j_0$, and thus 
 \[
 \left\|  A^{-i} A^{j} \right\| \le C~,
 \] for a constant $C$ independent of $i,j$. 
 
 The statement concerning the inhomogeneous case follows from these observations, and from the fact that every compact subset $K \subset \mathbb{R}^d$ is contained in 
 $\tilde{Q}_0^{i*}$, for sufficiently large $i$. 
\end{proof}

\begin{lemma}
For any two coverings $\mathcal{Q} = (A^j Q_0)_{j \in \mathbb{Z}}$ and $\mathcal{P}= (A^i P_0)_{i \in \mathbb{Z}}$ induced by the same matrix $A$, one has 
\[ \mathcal{D}(\mathcal{Q},{\rm L}^p,\ell^q_{v}) = \mathcal{D}(\mathcal{P},{\rm L}^p,\ell^q_{v}) ~.\]
The same statement holds for inhomogeneous coverings. 
\end{lemma}
\begin{proof}
The previous lemma shows that $\mathcal{Q}$ and $\mathcal{P}$ are equivalent. Hence in order to apply Lemma \ref{lem:suf_dc_equal},
it remains to show that $\|A^{j-i}\| < C$, for all pairs $i,j$ with $Q_j \cap P_i \not= 0$. But the proof of the previous lemma shows that $Q_j \cap P_i \not= \emptyset$ entails $|j-i|< j_0$, with a fixed $j_0$, hence the norm estimate holds as well. 
\end{proof}

We next want to identify (homogeneous and inhomogeneous) anisotropic Besov spaces as special cases of decomposition spaces. In the context of this paper, the chief purpose of this result is to make Theorem \ref{thm:rigidity} and \ref{lem:suf_dc_equal} available for the discussion of anisotropic Besov spaces. It is however of some independent interest, since it allows include the anisotropic Besov spaces in a unified view onto a large range of decomposition spaces (e.g., $\alpha$-modulation spaces, curvelet smoothness spaces, wavelet coorbit spaces, etc.), see \cite[Chapter 6]{VoDiss}. For the case of inhomogeneous Besov spaces associated to diagonal matrices, the following result is \cite{BorupNielsenDecomposition}, for isotropic Besov spaces, it can be found in \cite[Lemma 6.2.2]{VoDiss}. Our proof is an adaptation of the proof for the latter to the anisotropic setting.

In order to motivate the following proof, we rewrite the Besov space norm of a tempered distribution $f$ as a decomposition space norm of its Fourier transform:
\begin{eqnarray*}
  \| f \|_{\dot{B}_{p,q}^\alpha} & = & \left\| \left( \| f \ast \psi_j \|_{p} \right)_{j \in \mathbb{Z}} \right\|_{\ell^q_{v_{\alpha,A}}}
  \\ & = &  \left\| \left( \left\| \mathcal{F}^{-1} \left( \widehat{f} \cdot  \phi_j \right) \right\|_{p} \right)_{j \in \mathbb{Z}} \right\|_{\ell^q_{v_{\alpha,A}}}
\end{eqnarray*} where we use $\phi_j = \widehat{\psi}_j$. Provided this family is a BAPU of a suitable covering, the right-hand side becomes a decomposition space norm, which suggests that the Fourier transform induces an isomorphism between the two spaces. There is however one subtlety to consider: Note that the ``reservoir'' of candidates for elements in the decomposition spaces $\mathcal{D}(\mathcal{P},{\rm L}^p,\ell^q_{v_{\alpha,A}})$ consists of {\em distributions} on the open set $\mathcal{O} = \mathbb{R}^d \setminus \{ 0 \}$, whereas $\widehat{f}$ is {\em tempered}. Hence the remaining part of the proof consists mostly in showing that every element of the decomposition space is in fact (the restriction of) a tempered distribution. 

A first step in this direction is the following theorem, see \cite[Theorem 8.2]{Vo_Embed1}; observe also the remark following that Theorem.
\begin{theorem} \label{thm:embed_dc_tempered}
 Let $\mathcal{Q} = (Q_i)_{i \in I}$ denote an almost structured admissible covering of $\mathcal{O} \subset \mathbb{R}^d$ with standardization 
 with standardization $((T_i)_{i \in I}, (b_i)_{i \in I}, (Q_i')_{i \in I})$ and BAPU $(\varphi_i)_{i \in I}$. 
 Let $v$ denote a $Q$-moderate weight on $I$. For each $N \in \mathbb{N}_0$, 
 define 
 \[
  w^{(N)} = (w_i^{(N)})_{i \in I}~, w_i^{(N)} =  |{\rm det}(T_i)|^{1/p} \max \left\{ 1, \| T_i^{-1} \|^{d+1} \right\} \left[ \inf_{\xi \in Q_i} (1+|\xi|) \right]^{-N}~.
 \]
 Let $I_0 \subset I$ and
 assume that, for some $N \in \mathbb{N}$ one has $w^{(N)} / v \in \ell^1(I_0)$. Then the map
 \begin{eqnarray*}
  \Phi & : &  \mathcal{D}(\mathcal{Q},{\rm L}^p,\ell^q_v) \to \mathcal{S}'(\mathbb{R}^d)~ \\
  \Phi(f) & : & \mathcal{S}(\mathbb{R}^d) \ni g \mapsto \sum_{i \in I_0} \langle \varphi_i \cdot f, g \rangle
 \end{eqnarray*}
is well-defined.
\end{theorem}

 \begin{theorem} \label{thm:besov_as_decsp}
 Let $A$ denote an expansive matrix, and let $\mathcal{Q}_A$ denote a homgeneous covering induced by $A^T$. For $\alpha \in \mathbb{Z}$, define the weight
 \[
  v_{\alpha,A} : \mathbb{Z} \to \mathbb{R}^+~,v_{\alpha,A}(j) = |{\rm det}(A)|^{j \alpha}~.
 \]
 Denote by $\rho: \mathcal{S}'(\mathbb{R}^d) \to \mathcal{D}'(\mathbb{R}^d)$ the restriction map. Then $\rho \circ \mathcal{F}$
is a topological isomorphism 
 \[
 \rho \circ  \mathcal{F} : \dot{B}_{p,q}^\alpha(A) \to \mathcal{D}(\mathcal{Q}_A,{\rm L}^p,\ell^q_{v_{\alpha,A}})~.
 \]
 Similarly, if $\mathcal{Q}_A^i$ denote an inhomogeneous covering induced by $A^T$, then
  \[
 \rho \circ \mathcal{F} : {B}_{p,q}^\alpha(A) \to \mathcal{D}(\mathcal{Q}_A^i,{\rm L}^p,\ell^q_{v_{\alpha,A}})~.
 \]
is a topological isomorphism, as well. Here  $v_{\alpha,A}$ denotes the restriction of the weight for the homogeneous setting to $\mathbb{N}_0$.  
 \end{theorem}
\begin{proof}
We first consider the inhomogeneous case.
Let $\psi_0,\psi$ denote a pair of functions fulfilling (\ref{eqn:def_wv2_strong_ih}), and define $\varphi_j = \widehat{\psi}_j$ for $j \ge 0$. 
Then $(\varphi_j)_{j \in  \mathbb{N}_0}$ is a BAPU relative to the admissible covering $Q_j = \varphi_j^{-1}(\mathbb{C} \setminus \{ 0 \})$. In fact, the BAPU is $p$-admissible, since
\[
  |{\rm det}(T_j)|^{\frac{1}{t}-1} \| \mathcal{F}^{-1} \varphi_j \|_{{\rm L}^p} = \| \psi \|_{{\rm L}^p}
\] holds for all $j  >0$ and all $p \in (0,1]$, with $t=\min(p,1)$. Using this family to compute the decomposition space norm, we find that
\begin{eqnarray*}
 \| f \|_{{B}_{p,q}^\alpha (A)} & = & \left\| \left( \| f \ast \psi_j \|_{{\rm L}^p} \right)_{j \in \mathbb{N}_0} \right\|_{\ell^q_{v_{\alpha,A}}} \\
 & = & \left\| \left( \| \mathcal{F}^{-1}(\widehat{f}  \cdot  \varphi_j)  \|_{{\rm L}^p} \right)_{j \in \mathbb{N}} \right\|_{\ell^q_{v_{\alpha,A}}} \\
 & = & \| \widehat{f } \|_{\mathcal{D}(\mathcal{Q},{\rm L}^p,\ell^q_{v_\alpha,A})}~.
\end{eqnarray*}
Hence the Fourier transform, mapping $f$ to (the restriction to $C_c^\infty(\mathbb{R}^d)$ of) its Fourier transform, is an isometric embedding of  $\dot{B}_{p,q}^\alpha(A)$ into $\mathcal{D}(\mathcal{Q},{\rm L}^p,\ell^q_{v_{\alpha,A}})$, and it remains to show that it is onto. 
To this end, we consider the auxiliary map
\[
 \Phi : \mathcal{D}(\mathcal{Q},{\rm L}^p,\ell^q_{v_{\alpha,A}}) \to \mathcal{S}'(\mathbb{R}^d)~.
\] defined by 
\[
 (\Phi(f)) (g) = \sum_{j \ge 0} f(\varphi_j \cdot g)~.
\]
In order to apply Theorem \ref{thm:embed_dc_tempered}, we need an estimate for
\[
 w^{(N)}(j) = |\det(T_j)|^{1/p} \max \{ 1, \| T_j^{-1} \|^{d+1}\} \left[ \inf_{\xi \in Q_j} (1+|\xi|) \right]^{-N}
\] where $T_j = (A^T)^j$, for $j>0$. Picking $\epsilon >0$ with $\mathbf{B}_\epsilon(0) \cap Q_1 = \emptyset$, we get by Lemma \ref{lem:exp_norm} that
\[
 \inf_{\xi \in Q_j} |\xi| \ge C \lambda_-^j~,
\] whenever $\lambda_-$ is strictly between $1$ and the smallest eigenvalue modulus of $A$. In addition, $\sup_{j \ge 0} \| T_j ^{-1}\|$ is bounded, hence we obtain   
\[
 \frac{w^{(N)}(j)}{v_{\alpha,A}(j)} \le C' |\det(A)|^{j/p-\alpha j} \lambda_-^{-jN}~,
\] in particular, $\frac{w^{(N)}}{v_{\alpha,A}} \in \ell^1(\mathbb{N}_0)$ as soon as $|\det(A)|^{1/p-\alpha} < \lambda_-^N$. Hence $\Phi$ is well-defined. 

Now, given any $f \in  \mathcal{D}(\mathcal{Q},{\rm L}^p,\ell^q_{v_{\alpha,A}}) $ and every $g \in \mathcal{D}(\mathbb{R}^d)$, we have that $g = \sum_{j=0}^M g \ast \varphi_j$ for $M$ sufficiently large, and thus 
\[
 f(g) = \sum_{j=0}^M f(\varphi_j \cdot g) = (\Phi(f))(g)~. 
\]  Thus $\rho \circ \Phi$ is the identity operator on $\mathcal{D}(\mathcal{Q}, {\rm L}^p,\ell^q_{v_{\alpha,A}})$. 
Now for any $f \in \mathcal{D}(\mathcal{Q},{\rm L}^p,\ell^q_{v_{\alpha,A}}) $ we can define $u = \mathcal{F}^{-1}(\Phi(f))$, and obtain that $\rho(\widehat{u})$ = $f$. 

We nowturn to the homogeneous case. We choose a wavelet $\psi$ fulfilling the condition (\ref{eqn:def_wv2_strong}), and define $\varphi_j = \widehat{\psi}_j$. Just as above, $(\varphi_j)_{j \in  \mathbb{Z}}$ is a $p$-admissible BAPU relative to the admissible covering $Q_j = \varphi_j^{-1}(\mathbb{C} \setminus \{ 0 \})$, and as before
\begin{eqnarray*}
 \| f \|_{\dot{B}_{p,q}^\alpha (A)} & = & \left\| \left( \| f \ast \psi_j \|_{{\rm L}^p} \right)_{j \in \mathbb{Z}} \right\|_{\ell^q_{v_{\alpha,A}}} \\
 & = & \left\| \left( \| \mathcal{F}^{-1}(\widehat{f}  \cdot  \varphi_j)  \|_{{\rm L}^p} \right)_{j \in \mathbb{Z}} \right\|_{\ell^q_{v_{\alpha,A}}} \\
 & = & \| \widehat{f } \|_{\mathcal{D}(\mathcal{Q},{\rm L}^p,\ell^q_{v_\alpha,A})}~.
\end{eqnarray*}
Again, it remains to see that $\rho \circ \mathcal{F}$ is onto. This time, we need auxiliary mappings
\[
 \Phi_{1,2} : \mathcal{D}(\mathcal{Q},{\rm L}^p,\ell^q_{v_{\alpha,A}})) \to \mathcal{S}'(\mathbb{R}^d)~,
\]
with
\[
 \Phi_1 (f)(g) = \sum_{j \ge 0} f(\varphi_j \cdot g)~,
\] 
and $\Phi_2$ defined later. 
Just as in the inhomogeneous case it follows that $\Phi_1(f) \in \mathcal{S}'(\mathbb{R}^d)$ holds for all $f \in \mathcal{D}(\mathcal{Q},{\rm L}^p,\ell^q_{v_{\alpha,A}})$. 

The small frequencies require more work. 
Given any Schwartz function $g$ and $N \in \mathbb{N}$, let $P_N g$ denote the Taylor polynomial of $g$ around zero,
\[
 P_N(g)(\xi) = \sum_{|\alpha|<N} \frac{\partial^\alpha g(0)}{\alpha !} \xi^\alpha~.
\]
Note that here we used the notation $|\alpha| = \sum_{i=1}^d \alpha_i$  for the length of a multi-index $\alpha$, somewhat in conflict to our notation for the euclidean length. However, no serious confusion can arise from this in the following. 

We then define, for any $f \in \mathcal{D}(\mathcal{Q},{\rm L}^p,\ell^q_{v_{\alpha,A}})$ and $g \in \mathcal{S}(\mathbb{R}^d)$,
\begin{equation} \label{eqn:Phi2}
 \Phi_2(f)(g) = \sum_{j < 0} f(\varphi_j \cdot (g - P_N g))~. 
\end{equation}
Our aim is to show that, for $N$ sufficiently large, the right-hand side converges and yields a tempered distribution. 

As a first step towards convergence of the right-hand side, we write
\[
 \varphi_j^* = \sum_{i: Q_i \cap Q_j \not= \emptyset} \varphi_i~,
\]
which implies $\varphi_j \cdot \varphi_j^* = \varphi_j$, and thus
\begin{eqnarray*}
 \left| f (\varphi_j (g-P_N g))\right|  & = & \left| (\varphi_j \cdot f) (\varphi_j^* \cdot (g-P_N g) \right| \\
 & = & \left| \mathcal{F}^{-1}( \varphi_j \cdot f) (\mathcal{F}(\varphi_j^* \cdot (g-P_N g))\right| ~.
 \end{eqnarray*}
 Here we employed the definition of the Fourier transform of tempered distributions by duality. Note now that by assumption, the tempered distribution $\mathcal{F}^{-1}( \varphi_j \cdot f)$ is an ${\rm L}^p$-function, whereas $(\mathcal{F}(\varphi_j^* \cdot (g-P_N g))$ is a Schwartz function. Hence we can continue estimating
 \begin{eqnarray} \nonumber
 \left| f (\varphi_j (g-P_N g))\right| & \le &  \left\| \mathcal{F}^{-1}( \varphi_j \cdot f) \right\|_{\infty} \left\|   (\mathcal{F}(\varphi_j^* \cdot (g-P_N g))\right\|_1 \\
 \nonumber & \le &  \left\| \mathcal{F}^{-1}( \varphi_j \cdot f) \right\|_{\infty} \left\|_\infty   (\mathcal{F}(\varphi_j^* \cdot (g-P_N g))\right\|_1 \\
 \label{eqn:partial_holder} & \le & C  |{\rm det}(A)|^{j/p} \left\| \mathcal{F}^{-1}( \varphi_j \cdot f) \right\|_{p} \left\|   (\mathcal{F}(\varphi_j^* \cdot (g-P_N g))\right\|_1 ~,
 \end{eqnarray}
 with the last estimate due to \cite[Lemma 5.3]{Vo_Embed1}, furnishing a constant $C$ that is independent of $j$ and $f$. 
We now sum the terms from (\ref{eqn:partial_holder}) and get 
 \begin{eqnarray}
 \nonumber
 \lefteqn{\sum_{j < 0} \left| f (\varphi_j (g-P_N g))\right| \le } \\ \nonumber & \le & C \sum_{j<0}  |{\rm det}(A)|^{j/p} \left\| \mathcal{F}^{-1}( \varphi_j \cdot f) \right\|_{p} \left\|   (\mathcal{F}(\varphi_j^* \cdot (g-P_N g))\right\|_1 \\
 & \le & C \sum_{j<0} \left( v_{\alpha,A}(j)  \left\| \mathcal{F}^{-1}( \varphi_j \cdot f) \right\|_{p} \right) \left(   |{\rm det}(A)|^{j/p} v_{\alpha,A}(j)^{-1}  \left\|   (\mathcal{F}(\varphi_j^* \cdot (g-P_N g))\right\|_1 \right) \nonumber \\ & \le & 
 \| f \|_{\mathcal{D}(\mathcal{Q},{\rm L}^p, \ell^q_{v_{\alpha,A}})} \left\|  \left(   |{\rm det}(A)|^{j/p} v_{\alpha,A}(j)^{-1}  \left\|   (\mathcal{F}(\varphi_j^* \cdot (g-P_N g))\right\|_1 \right)_{j < 0} \right\|_{\ell^{\infty}} \label{eqn:zwischen}~.
\end{eqnarray}
This puts the norms $ \left\|   (\mathcal{F}(\varphi_j^* \cdot (g-P_N g))\right\|_1$ in the focus of our attention. Here the usual estimates relating decay of the Fourier transform to norms of the derivatives provides
\[
 \|  (\mathcal{F}(\varphi_j^* \cdot (g-P_N g)) \|_1 \le M \sum_{|\alpha| \le d+1} \left\| \partial^\alpha \left[ \varphi_j^* \cdot (g-P_N g) \right] \right\|_1~,
\] see e.g. \cite[Lemma 3.5]{Fu_coorbit}. 
Using the Leibniz formula for derivatives of products yields
\begin{equation} \label{eqn:leibniz}
 \partial^\alpha \left[ \varphi_j^* \cdot (g-P_N g) \right](\xi) = \sum_{\beta + \gamma = \alpha} \left( \begin{array}{c} \alpha \\ \beta \end{array} \right) \partial^\beta (\varphi_j^*) (\xi) 
 \partial^\gamma (g-P_N(g))(\xi) ~.
\end{equation}

By construction of the $\varphi_j$, we have $\varphi_j^* (\xi) = \varphi_0^*((A^T)^j \xi)$, for all $j < 0$. Choose $C_{\varphi,1} >0$ large enough such that ${\rm supp}(\varphi_0^*)$ is contained in the ball of radius $R$, then we get via Lemma \ref{lem:exp_norm}
\begin{equation} \label{eqn:supp_phi_est}
 \forall \xi \in {\rm supp}(\varphi_j^*) ~:~|\xi| \le C_{\varphi,1} \lambda_-^{j}~,
\end{equation}  where $\lambda_->1$ is a lower bound for the eigenvalues of $A$. 
Furthermore, the chain rule and the fact that $\varphi_j^*$ is a dilation of $\varphi_0^*$ allows to estimate
\begin{equation} \label{eqn:part_phi}
 \left| \partial^\alpha \varphi_j^*(\xi) \right| \le C_{\varphi,2} (1+\|h \|_\infty)^{|\alpha|}~.
\end{equation}

Given a multi-index $\gamma$ of order $\le d$, all partial derivatives of $\partial^\gamma (g-P_N g)$ of order less than $N-d$ vanish at zero. Hence 
Taylor's formula allows to estimate 
\begin{equation} \label{eqn:taylor}
 \left| \partial^\gamma (g-P_N g) (\xi) \right| \le C_d |\xi|^{N-d} \underbrace{\sum_{|\alpha| \le N} \| \partial^\alpha g \|_{\infty}}_{C_{N,g}}~,  
\end{equation}
for all $\xi \in \mathbb{R}^d$. 

We can now employ the collected estimates to show unconditional convergence of the right-hand side of (\ref{eqn:Phi2}). Combining (\ref{eqn:part_phi}) with (\ref{eqn:taylor}) gives
\begin{equation}
  \left| \partial^\beta (\varphi_j^*) (\xi)  \partial^\gamma (g-P_N(g))(\xi) \right| \le C_d C_{\varphi,2} (1+\|h \|_\infty)^{|\alpha|} C_{N,g} |\xi|^{N-d}~,
\end{equation}
and on the support of this pointwise product, we can employ (\ref{eqn:supp_phi_est}), to get finally 
\begin{equation}
 \label{eqn:final_1}
\left| \partial^\beta (\varphi_j^*) (\xi)  \partial^\gamma (g-P_N(g))(\xi) \right| \le  C_d C_{\varphi,1} C_{\varphi,2} C_{N,g} \lambda_-^{j(N-d)} ~.
\end{equation}
A further consequence of (\ref{eqn:supp_phi_est}), we have that all $\varphi_j^*$, for $j< 0$, are supported in the ball of radius $C_{\varphi,1}$, hence integrating (\ref{eqn:final_1}) yields
\begin{equation}
 \label{eqn:final_2}
\left\| \partial^\beta (\varphi_j^*)  \partial^\gamma (g-P_N(g))(\xi) \right\|_1 \le  C' C_{\varphi,1}^{d+1} C_{\varphi,2} C_{N,g} \lambda_-^{j(N-d)}
\end{equation}
Hence the triangle inequality applied to (\ref{eqn:leibniz}) yields that
\begin{equation} \label{eqn:final_3}
 \left\| \partial^\alpha \left[ \varphi_j^* \cdot (g-P_N g) \right] \right\| \le C'' C_{N,g} \lambda_-^{j(N-d)}~,
\end{equation}
with the constant $C''$ aggregating the constants $C_d,C_{\varphi,1},C_{\varphi,2}$, and the coefficients entering in the sum (\ref{eqn:leibniz}); observe that the latter are independent of $j$ and $N$. 

This yields, for all $j < 0$,
\begin{eqnarray*}
  {\rm det}(A)|^{j/p} v_{\alpha,A}(j)^{-1}  \left\|   (\mathcal{F}(\varphi_j^* \cdot (g-P_N g))\right\|_1  & \le &  C' C_{N,g} |{\rm det}(A)|^{j(1/p-\alpha)} \lambda_-^{j(N-d)}
\end{eqnarray*}
and this expression can be uniformly bounded in $j$ as soon as $\lambda_-^{N-d} > |{\rm det}(A)|^{\alpha-1/p}$. 
But then we get by (\ref{eqn:zwischen}) that 
\[
 \sum_{j < 0} \left| f (\varphi_j (g-P_N g))\right| \le C''' C_{N,g}~,
\] which yields, firstly, that the sum defining $\Phi_2(f)(g)$ converges unconditionally for all Schwartz functions $g$, and secondly, that the linear map $g \mapsto \Phi_2(f)(g)$ is indeed a tempered distribution. 

Hence $\Phi(f) = \Phi_1(f) + \Phi_2(f)$ defines a tempered distribution, and the fact that $P_N g = 0$ for all $g \in C_c^\infty(\mathbb{R}^d \setminus \{ 0 \})$ yields that
$\Phi(f)(g) = f(g)$ for all $f \in \mathcal{D}(\mathcal{Q},{\rm L}^p,\ell^q_{v_{\alpha,A}})$. Now the same argument as for the inhomogeneous case allows to conclude from this that $\rho \circ \mathcal{F}$ is onto. 
\end{proof}

Thus the theory of decomposition spaces becomes available for the study of anisotropic Besov spaces, which puts the role of the induced coverings into focus.
The following lemmata transfer the comparison of induced coverings to the comparison of associated homogeneous quasi-norms. We begin with a characterization of weak equivalence for induced coverings. 
\begin{lemma} \label{lem:char_we1}
 Let $A$ and $B$ be two expansive matrices. 
\begin{enumerate}
 \item[(a)] The homogeneous coverings induced by $A$ and $B$ are weakly equivalent if and only if, for all $R>0$, one has 
    \begin{eqnarray} \label{eqn:ub1}
    \sup_{i \in \mathbb{Z}} & & \left| \left\{ j \in \mathbb{Z} : \| A^{-j} B^i \| \ge R^{-1} \mbox{ and } \| B^{-i} A^j \| \ge R^{-1} \right\} \right| < \infty~,\\
   \label{eqn:ub2} \sup_{j \in \mathbb{Z}} & & \left| \left\{ i \in \mathbb{Z} : \| A^{-j} B^i \| \ge R^{-1} \mbox{ and } \| B^{-i} A^j \| \ge R^{-1} \right\} \right| < \infty~,
 \end{eqnarray}
 \item[(b)] The inhomogeneous coverings induced by $A$ and $B$ are weakly equivalent if and only if, for all $R>0$, one has 
   \begin{eqnarray} \label{eqn:iub1}
    \sup_{i \in \mathbb{N}_0} & & \left| \left\{ j \in \mathbb{N}_0 : \| A^{-j} B^i \| \ge R^{-1} \mbox{ and } \| B^{-i} A^j \| \ge R^{-1} \right\} \right| < \infty~,\\
   \label{eqn:iub2} \sup_{j \in \mathbb{N}_0} & & \left| \left\{ i \in \mathbb{N}_0 : \| A^{-j} B^i \| \ge R^{-1} \mbox{ and } \| B^{-i} A^j \| \ge R^{-1} \right\} \right| < \infty~,
 \end{eqnarray}
\end{enumerate} 
\end{lemma}
\begin{proof}
For the proof of $(a)$, first assume that (\ref{eqn:ub1}) and (\ref{eqn:ub2}) hold for all $R>0$. 
 Fix $S>0$ large enough, so that $\bigcup_{j \in \mathbb{Z}} A^j C_S = \bigcup_{i \in \mathbb{Z}} B^i C_S = \mathbb{R}^d \setminus \{ 0 \}$. It is then sufficient to show that the coverings
 $(A^j C_S)_{j \in \mathbb{Z}}$ and $(B^i C_S)_{i \in \mathbb{Z}}$ are weakly equivalent. Let $K$ denote a finite upper bound for the suprema in (b), for $R=S^2$. Given $i,j \in \mathbb{Z}$, one then has that
 $i \in I_j$ if and only if $A^j C_S \cap B^i C_S \not= \emptyset$, or equivalently, if and only if $B^{-i} A^j C_S \cap C_S \not= \emptyset$. This implies the existence of $x \in C_S$ such that $B^{-i} A^j x \in C_S$. 
 Hence 
 \[
  \| A^{-j} B^i \| \ge \frac{ | x | }{| B^{-i} A^j x |} \ge \frac{1}{S^2}~. 
 \] But the choice of $K$ yields that for any given $j \in \mathbb{Z}$, there are at most $K$ indices $I$ with $\| A^{-j} B^i \| \ge 1/S^2$, and thus we get 
 \[
  \sup_{j \in \mathbb{Z}} |I_j| \le K~.
 \] By symmetry, we obtain the second inequality, hence the coverings are weakly equivalent.
 
 Now assume that $(A^j C_R)_{j \in \mathbb{Z}}$ and $(B^i C_R)_{i \in \mathbb{Z}}$ are weakly equivalent. Fix $j \in \mathbb{Z}$, and assume that $\| B^{-i} A^j \| \ge R^{-1}$ as well as $\| A^{-j} B^i \| \ge R^{-1}$ hold for some $i \in \mathbb{Z}$.  We aim to show that $i \in I_j$, then the upper bound on $|I_j|$ provided by the assumption of weak equivalence yields (\ref{eqn:ub1}).
 
 The first inequality yields $z_1 \in \mathbb{R}^n$ with $|z_1| = 1$ and $|A^{-j} B^i z_1| \ge R^{-1}$. Now if $|A^{-j} B^i z_1| < R$, then we have $A^{-j} B^i z_1 \in A^{-j} B^i C_R \cap C_R$, and thus $i \in I_j$. 
 
 Hence assume that $|A^{-j} B^i z_1| \ge R$. We use the inequality $\| A^{-j} B^i \| \ge R^{-1}$ to conclude the existence of $y \in \mathbb{R}^d$ with $|y| =1$ and $|B^{-i} A^j y| >R^{-1}$. If $|B^{-i} A^j y| < R$, we find $B^{-i} A^j y \in C_R \cap B^{-i} A^j C_R$, and thus $i \in I_j$. 
 
 Hence the remaining case is $|y| = |z_1| = 1$ with $|B^{-i} A^j y|\ge R$ and $|A^{-j} B^i z_1|\ge R$. Define $z_2 = \frac{1}{|B^{-i} A^j y|}B^{-i} A^j y$, hence $|z_2| = 1$. Pick a continuous curve $\varphi :[0,1] \to \{ x \in \mathbb{R}^n : |x| = 1 \}$ with $\varphi(0) = z_1$, $\varphi(1) = z_2$, then the function
 \[
  \widetilde{\varphi}: [0,1] \to \mathbb{R}^+~,t \mapsto |A^{-j} B^i \varphi(t)|
  \] is continuous, with 
  \[
   \widetilde{\varphi}(0) = |A^{-j} B^i z_1| \ge R~,~\widetilde{\varphi}(1) = \frac{|A^{-j} B^iB^{-i} A^j y|}{|B^{-i} A^j y|} \le R~.
  \] Hence the intermediate value theorem yields $t$ with $\widetilde{\varphi}(t) = 1$. It follows that  $\varphi(t) \in C_R$ as well as $ A^{-j} B^i \varphi(t) \in C_R$, and thus
  $A^j C_R \cap B^i C_R \not= \emptyset$. Hence this case also leads to $i \in I_j$. 
  
  Hence the uniform upper bound for all $|I_j|$ yields (\ref{eqn:ub1}). In the same way, we obtain (\ref{eqn:ub2}) from an upper bound on all $|J_i|,~i \in \mathbb{Z}$. 
  
  The inhomogeneous case (b) follows entirely analogously. 
\end{proof}

\begin{lemma} \label{lem:char_equiv_matr}
 Let $A$ and $B$ be two expansive matrices, and $\mathcal{Q}$ and $\mathcal{P}$ coverings induced by $A$ and $B$, respectively. Let $\rho_{A}$ denote an $A$-homogeneous quasi-norm and $\rho_{B}$ a $B$-homogeneous quasi-norm. Then the following are equivalent:
 \begin{enumerate}
  \item[(a)] The homogeneous coverings $\mathcal{Q}$ and $\mathcal{P}$ are weakly equivalent. 
  \item[(b)] The homogeneous coverings $\mathcal{Q}$ and $\mathcal{P}$ are equivalent.
  \item[(c)] The quasi-norms $\rho_{A}$ and $\rho_{B}$ are equivalent. 
 \end{enumerate}
\end{lemma}
\begin{proof}
For the proof of $(a) \Rightarrow (c)$, we show the contraposition. Hence assume that the quasi-norms $\rho_{A}$ and $\rho_B$ are not equivalent. Then Lemma \ref{lem:norm_equiv} (a) $\Leftrightarrow (c)$ yields a sequence $(k_n)_{n \in \mathbb{Z}}$ of $k_n$ such that 
\[
  \left\| A^{k_n} B^{-\lfloor \epsilon k_n \rfloor} \right\| \to \infty~.
\] as $n \to \infty$. 
By the choice of $\epsilon$, we have 
\begin{eqnarray*}
 |{\rm det}(B^{\lfloor \epsilon k \rfloor} A^{-k})| & = & |{\rm det}(A)|^{-k} |{\rm det}(B)|^{\epsilon k} |{\rm det}(B)|^{-\epsilon k + \lfloor \epsilon k \rfloor}  \\
  & = & |{\rm det}(B)|^{-\epsilon k + \lfloor \epsilon k \rfloor} \ge |{\rm det}(B)|^{-1} ~,
\end{eqnarray*} for all $k \in \mathbb{Z}$.
Hence Lemma \ref{lem:norm_est} implies that 
\[
  \left\|  B^{\lfloor \epsilon k_n \rfloor} A^{-k_n}\right\| \to \infty~
\] as well. 

Now fix $R>1$, $m \in \mathbb{N}$, and pick $k_n$ with 
\[
 \left\| A^{k_n} B^{-\lfloor \epsilon k_n \rfloor} \right\| \ge R^{-1} \max(\|B \|,\|B^{-1} \|)^m~
\]
as well as 
\[
  \left\|  B^{\lfloor \epsilon k_n \rfloor} A^{-k_n}\right\| \ge R^{-1} \max(\|B \|,\|B^{-1} \|)^m~.
\]
Using the norm estimate $\| S T \| \le \| S \| \| T \|$ for arbitrary matrices $S,T$ then gives for $i=0,\ldots, m$ that 
\[
 \left\| A^{k_n}  B^{i-\lfloor \epsilon k_n \rfloor} \right\| \ge \| A^{k_n}  B^{-\lfloor \epsilon k_n \rfloor} \| \| B \|^{-i} \ge R^{-1}
\]
as well as 
\[
 \left\|  B^{\lfloor \epsilon k_n \rfloor-i} A^{-k_n}\right\| \ge \| A^{k_n}  B^{-\lfloor \epsilon k_n \rfloor} \| \| B^{-1} \|^{-i} \ge R^{-1}~. 
\]
But this means that condition (\ref{eqn:ub2}) is violated, and thus the induced coverings are not weakly equivalent. 

Now assume that $\rho_A$ and $\rho_B$ are equivalent. Let $\mathcal{Q}$ denote a covering induced by $A$. 
Define
\[ 
 \mathbf{B}^A_r(0) = \{ x  \in \mathbb{R}^d~:~ \rho_A(x) < r \}~,
\] the ball with respect to $\rho_A$ with center $0$ and radius $r$. 

The fact that $Q_0$ has compact closure in  $\mathbb{R}^d \setminus \{ 0 \}$ then yields $Q_0 \subset B^A_R(0) \setminus B^A_{R^{-1}}(0)$. By construction of the covering on the one hand, and $A$-homogeneity of $\rho_A$ on the other, this entails
\[
 Q_j = A^j Q_0  \subset \mathbf{B}^A_{R |{\rm det}(A)|^{j}}(0) \setminus \mathbf{B}^A_{R^{-1}  |{\rm det}(A)|^{j}}(0)~.
\]

By analogous reasoning, we get (possibly after increasing $R$) that also 
\[
 P_i \subset \mathbf{B}^B_{R |{\rm det}(B)|^{i}}(0) \setminus \mathbf{B}^B_{R^{-1}  |{\rm det}(B)|^{i}}(0)~,
\] with $\mathbf{B}^B_R(0)$ denoting balls with respect to $\rho_B$.

By assumption, there exists $c\ge 1$ such that
\[
 \frac{1}{c} \rho_A(x) \le \rho_B(x) \le c \rho_A(x)~.
\]

Now let $i, j \in \mathbb{Z}$ with 
\begin{eqnarray*}
\emptyset & \not= &  Q_j \cap P_i  \\ 
 & \subset &   \left(  \mathbf{B}^A_{R |{\rm det}(A)|^{j}}(0) \setminus \mathbf{B}^A_{R^{-1}  |{\rm det}(A)|^{j}}(0)\right)  \cap 
 \left( \mathbf{B}^B_{R |{\rm det}(B)|^{i}}(0) \setminus \mathbf{B}^B_{R^{-1}  |{\rm det}(B)|^{i}}(0) \right)~.
\end{eqnarray*}
For any $x$ contained in this intersection, one obtains in particular that
\[
 R^{-1} |{\rm det}(A)|^j \le \rho_A(x) \le c \rho_B(x) \le c R |{\rm det}(B)|^i ~,
\]
which leads to 
\[
 \frac{|{\rm det}(A)|^j}{|{\rm det}(B)|^i} \le c R^2~.
\] But analogous reasoning also yields
\[
 \frac{|{\rm det}(B)|^i}{|{\rm det}(A)|^j} \le c R^2~.
\] 
Using $\epsilon = \frac{\ln(|{\rm det}(A)|)}{\ln (|{\rm det}(B)|)}$ as in Lemma \ref{lem:norm_equiv}, the two equations yield 
\[ |j \epsilon - i| \le \frac{\ln(c R^2)}{\ln(|\det(B)|)} ~.\]
Thus, with $i_0 = \lfloor j \epsilon \rfloor$ and $K = \lceil \frac{\ln(c R^2)}{\ln(|{\rm det}(B))|} \rceil +1$, we get
\begin{eqnarray*}
 Q_j & = & A^j Q_0 \subset \bigcup_{\ell=i_0-K}^{i_0+K} \mathbf{B}_{R|{\rm det}(B)|^\ell}(0)\setminus \mathbf{B}_{R^{-1}|{\rm det}(B)|^\ell}(0)  \\
 & = & B^{i_0} \left( \bigcup_{\ell = -K}^K   \mathbf{B}_{R|{\rm det}(B)|^\ell}(0)\setminus \mathbf{B}_{R^{-1}|{\rm det}(B)|^\ell}(0)   \right)~. 
\end{eqnarray*}
Since 
\[
\overline{ \left( \bigcup_{\ell = -K}^K  \mathbf{B}_{R|{\rm det}(B)|^\ell}(0)\setminus \mathbf{B}_{R^{-1}|{\rm det}(B)|^\ell}(0)   \right)} \subset \mathbb{R}^d \setminus \{ 0 \}
\] is compact, there exists $k \in \mathbb{N}$ such that
\[
\left( \bigcup_{\ell = -K}^K  \mathbf{B}_{R^{-1}|{\rm det}(B)|^\ell}(0) \right) \subset P_0^{k*}
\] and therefore
\[
 Q_j \subset B^{i_0} P_0^{k*} = P_{i_0}^{k*}~.
\]
Exchanging the roles of $Q_j$ and $P_i$ yields a converse inclusion relation, and we have shown 
equivalence of the induced coverings. 

Finally, $(b) \Rightarrow (a)$ is trivial. 
\end{proof}

The following is an analogy for the inhomogeneous setting. The proof is straightforward adaptation of the previous one, and therefore omitted.  
\begin{lemma} \label{lem:char_equiv_matr_ih}
 Let $A$ and $B$ be two expansive matrices, and $\mathcal{Q}$ and $\mathcal{P}$ inhomogeneous coverings induced by $A$ and $B$, respectively. Let $\rho_{A}$ denote an $A$-homogeneous quasi-norm and $\rho_{B}$ a $B$-homogeneous quasi-norm. Then the following are equivalent:
 \begin{enumerate}
  \item[(a)] The inhomogeneous coverings $\mathcal{Q}$ and $\mathcal{P}$ are weakly equivalent. 
  \item[(b)] The inhomogeneous coverings $\mathcal{Q}$ and $\mathcal{P}$ are equivalent.
  \item[(c)] The quasi-norms $\rho_{A}$ and $\rho_{B}$ are coarsely equivalent. 
 \end{enumerate}
\end{lemma}

We can now transfer these findings to the level of Besov spaces. First a rigidity theorem: Two matrices are Besov-equivalent if and only if the associated scale of Besov spaces coincide {\em in one nontrivial instance}:

\begin{theorem} \label{thm:rigidity_besov}
Let $A,B$ denote expansive matrices.
\begin{enumerate}
 \item[(a)] $A \sim_{\dot{B}} B$ holds if and only if there exists a tuple $(p,q) \not= (2,2,)$ and $\alpha \in \mathbb{R}$ such that 
 \[
  \dot{B}_{p,q}^{\alpha}(A) = \dot{B}_{p,q}^{\alpha}(B)~.
 \]
 \item[(b)] $A \sim_{{B}} B$ holds if and only if there exists a tuple $(p,q) \not= (2,2)$ and $\alpha \in \mathbb{R}$ such that 
 \[
  {B}_{p,q}^{\alpha}(A) = {B}_{p,q}^{\alpha}(B)~.
 \]
\end{enumerate}
\end{theorem}

\begin{proof}
In both cases, we need to show the ``if'' part of the statement. 
Assuming $\dot{B}_{p,q}^{\alpha}(A) = \dot{B}_{p,q}^{\alpha}(B)$ for one tuple $(p,q) \not= (2,2)$, we have that the homogeneous coverings $\mathcal{Q} = (Q_j)_{j \in \mathbb{Z}}$ and $\mathcal{P} = (P_i)_{i \in \mathbb{Z}}$ induced by $A^T$ and $B^T$, respectively, must be weakly equivalent. Now Lemma \ref{lem:char_equiv_matr} yields that the coverings are strongly equivalent, and also, that the induced quasi-norms are equivalent. This also implies $v_{\alpha,A} \asymp v_{\alpha,B}$. Finally, the equivalence of the homogeneous quasi-norms entails that $Q_j \cap P_i \not= \emptyset$ implies $j \in i_0 + \{ -k,\ldots,k \}$, where  $i_0 = \lfloor j \epsilon \rfloor$, and $k$ is independent of $j$; see the proof of Lemma \ref{lem:char_equiv_matr}. But this yields a uniform upper bound on $\| A^{-j} B^i \|$, via Lemma \ref{lem:norm_equiv}(c). Now Lemma \ref{lem:suf_dc_equal} becomes applicable and shows that 
 \[
  \mathcal{D}(\mathcal{Q},{\rm L}^p,\ell^q_{v_\alpha,A}) = \mathcal{D}(\mathcal{P},{\rm L}^p,\ell^q_{v_\alpha,B})~.
 \] Finally, Theorem \ref{thm:besov_as_decsp} translates this statement to the level of Besov spaces. 
 
 The proof of (b) is entirely analogous, noting that coarse equivalence of the norms is enough to guarantee that $v_{\alpha,A} \asymp v_{\alpha,B}$ holds on $\mathbb{N}_0$. 
\end{proof}

Finally, we record a handy characterization of Besov equivalence: 
\begin{corollary} \label{cor:char_equiv}
 Let $A,B$ denote expansive matrices. 
 \begin{enumerate}
  \item[(a)] $A \sim_{\dot{B}} B$ if and only if the $A^T$- and $B^T$-homogeneous quasi-norms are equivalent. 
  \item[(b)] $A \sim_B B$ if and only if the $A^T$- and $B^T$-homogeneous quasi-norms are coarsely equivalent. 
 \end{enumerate} 
In particular, $A \sim_{\dot{B}}B$ implies $A \sim_B B$.  
 \end{corollary}
%
%

\begin{remark} \label{rem:rel_Hardy_eq}
 Our result provides an interesting contrast to the results of Bownik \cite{Bow03}, who studied the analogous question for anisotropic Hardy spaces. Defining $A \sim_H B$ by the requirement that the anisotropic Hardy spaces induced by $A$ and $B$ coincide, he obtained that $A \sim_H B$ if and only if $\rho_A$ and $\rho_B$ are equivalent \cite[Theorem 10.5]{Bow03}. Hence, in view of Remark \ref{rem:transpose}, we find that $A \sim_H B$ if and only if $A \sim_{\dot{B}} B$. 
\end{remark}

\section{Characterizing (coarse) equivalence of matrices}

\label{sect:char_equiv}

It remains to derive explicit and checkable criteria for the (coarse) equivalence of matrices. 
We first derive necessary criteria in terms of generalized eigenspaces, based on an approach developed in \cite{Bow05}. This requires the following class of auxiliary subspaces. 
\begin{definition}
 Given a matrix $A \in \mathbb{C}^{d \times d}$, we define for $r >0$ and $m \in \mathbb{N}_0$ 
 \[
  E(A,r,m) = {\rm span} \left(  \bigcup_{|\lambda| = r} {\rm Ker}(A-\lambda I_d)^m \cup \bigcup_{|\lambda|<r}   {\rm Ker}(A-\lambda I_d)^d  \right)~.
 \]
\end{definition}

The significance of these auxiliary spaces becomes apparent by the following lemma, which characterizes them by the asymptotic behaviour of $|A^k z|$, as $k \to \infty$. For a proof, see \cite[Lemma (10.4)]{Bow03}:
\begin{lemma} \label{lem:norm_growth}
 Let $A \in \mathbb{C}^{d \times d}$. For any $z \in \mathbb{C}^n \setminus \{ 0 \}$, $r>0$ and $m \in \mathbb{N}_0$, the condition
 \[
  z \in E(A,r,m+1) \setminus E(A,r,m) 
 \] is equivalent to the existence of a constant $c>0$ and $k_0 \in \mathbb{N}$ such that 
 \[
  \forall k \ge k_0~:~ \frac{1}{c} k^m r^k \le |A^k z| \le c k^m r^k~. 
 \]
\end{lemma}

We can now give necessary criteria in terms of generalized eigenspaces. 
\begin{lemma} \label{lem:char_eigenspaces}
Let $A,B$ denote expansive matrices. 
\begin{enumerate}
\item[(a)] If $A$ and $B$ are coarsely equivalent, then 
for all $r>0$ and all $m \in \mathbb{N}$
 \[
  E(A^{-1},r^\epsilon,m) = E(B^{-1},r,m)~.
 \]
\item[(b)] If $A$ and $B$ are equivalent, then for all $r>1$ and all $m \in \mathbb{N}$:
 \[
  {\rm span}\left( \bigcup_{|\lambda| = r^\epsilon} {\rm Ker}(A-\lambda I_d)^m \right)  =  {\rm span}\left( \bigcup_{|\lambda| = r} {\rm Ker}(B-\lambda I_d)^m \right) ~.
 \]
\end{enumerate}
\end{lemma}
\begin{proof}
 For the proof of (a) assume that $A$ and $B$ are coarsely equivalent. By Lemma \ref{lem:char_cequiv_matr1}, this is equivalent to 
 \[
\sup_{k \in \mathbb{N}}  \| A^{-k} B^{\lfloor \epsilon k \rfloor} \| < \infty ~.
 \]
 Note that Lemma \ref{lem:norm_est} then also provides
 \[
  \sup_{k \in \mathbb{N}} \| B^{-\lfloor \epsilon k \rfloor} A^k \| < \infty ~,
 \]
and the two estimates provide constants $0 < c_1 \le c_2 < \infty$ such that, for all $z \in \mathbb{C}^n \setminus \{ 0 \}$, and all $k \in \mathbb{N}$,
\begin{equation}  \label{eqn:norm_equiv_aux}
 c_1 \le \frac{|A^{-k} z|}{|B^{-\lfloor \epsilon k \rfloor} z|} \le c_2~. 
\end{equation}

 Now assume that $z \in E(A^{-1},r^\epsilon,m+1) \setminus E(A^{-1},r^\epsilon,m)$. Then Lemma \ref{lem:norm_growth} yields the existence of $c>0$ and $k_0 \in \mathbb{N}$ such that
 \[
  \forall k \ge k_0~:~ \frac{1}{c} k^m r^{\epsilon k} \le |A^{-k} z| \le c k^m r^{\epsilon k}~. 
 \]
This entails, via (\ref{eqn:norm_equiv_aux}), that 
\[
  \frac{c_1}{c} k^m r^{\epsilon k} \le |B^{-\lfloor \epsilon k \rfloor}| \le c_2 c k^m r^{\epsilon k}~, 
\] and thus, since 
\[
 r \le r^{\epsilon k - \lfloor \epsilon k \rfloor} \le 1~,
\]
we obtain, with a new constant $c'>0$, and for all $\ell \ge \lceil k_0 \epsilon \rceil$, that 
\begin{equation} \label{eqn:almost_asymp}
\frac{1}{c'} \ell^m r^{\ell} \le |B^{-\ell} z| \le c' \ell^m r^\ell~,
\end{equation}
as long as $\ell \in M = \{ \lfloor \epsilon k \rfloor : k \ge k_0 \}$. Now let $\ell \ge \lfloor \epsilon k_0 \rfloor$ be arbitrary. Then there exists $\ell_1 \le \ell$ and $j \in \{ 0,\ldots, \lceil 1/\epsilon \rceil \}$ such that $\ell_1 \in M$ and $\ell = \ell_1 + j$.  Assuming in addition that $\ell \ge \ell_0 = \max(2 \lceil 1/\epsilon, \rceil,\lfloor \epsilon k_0 \rfloor)$, we obtain the estimates
\begin{eqnarray*}
 \frac{1}{c'} \ell^m r^\ell & = & \frac{1}{c'} (\ell_1+j)^m r^{\ell_1+j}  \\
 & \le & \frac{1}{c'} 2^m r^j \ell_1^m r^{\ell_1} \le  2^m |B^{-\ell_1} z| \\
  & \le &  2^m \max_{0 \le j \le \lceil 1/\epsilon \rceil} \| B^j \|~ |B^{-\ell} z|~,
\end{eqnarray*} where we also used that since $A$ is expansive, the assumption that $z \in E(A^{-1},r^\epsilon,m+1) \setminus \{ 0 \}$ forces $r < 1$.  
By a similar calculation, we obtain
\[
 c' \ell^m r^\ell \ge r^{\lceil 1/\epsilon \rceil} \min_{0 \le j \le \lceil 1/\epsilon \rceil} \| B^{-j} \|^{-1} ~ |B^{-\ell} z|~,
\] and thus we have shown that (\ref{eqn:almost_asymp}), with different constant $c''$ instead of $c'$, holds for all $\ell \ge \ell_0$. But then Lemma \ref{lem:norm_growth} entails $z \in E(B^{-1},r,m+1) \setminus E(B^{-1},r,m)$.

The same argument with $A,B$ interchanged yields $E(B^{-1},r,m+1) \setminus E(B^{-1},r,m) \subset  E(A^{-1},r^\epsilon,m+1) \setminus E(A^{-1},r^\epsilon,m)$, and thus equality of the two difference sets. Now induction, first over $m$, then over the eigenvalues of $A^{-1}$ in increasing order, yields
$E(A^{-1},r^\epsilon,m) =   E(B^{-1},r,m)$. 

For the proof of (b) we observe that the above argument can also be applied to $k <0$, and then it yields
$E(A,r^{\epsilon},m)  = E(B,r,m)$ as well, for all $r,m$. Now, finally, the observation that
\[
E(B^{-1},r^{-1},m) \cap E(B,r,m) =  {\rm span}\left( \bigcup_{|\lambda| = r} {\rm Ker}(B-\lambda I_d)^m \right)\,\,
\] and the analogous fact about $A$ yields the conclusion of (b).
\end{proof}

\begin{remark} \label{rem:counter_bownik}
Part (b) of the previous lemma was noted by Bownik, see \cite[Theorem (10.3)]{Bow03}, and our proof is essentially an adaptation of the proof given there. The cited theorem also states the converse, i.e., that the necessary condition in (b) is sufficient as well. This is {\em false}, as can be seen with the help of the pair $A,B$ of matrices given by 
\[
A = \left( \begin{array}{cc} 2 & 2 \\ 0 & 2 \end{array} \right) \,\, ,\,\, B =  \left( \begin{array}{cc} 2 & 4 \\ 0 & 2 \end{array} \right)\,.
\] These matrices fulfill $\epsilon(A,B)=1$ and the necessary condition of part (b). However, one readily computes for all $k \in \mathbb{N}$ that
\[
A^{-k} B^k =  \left( \begin{array}{cc} 1 & k \\ 0 & 1 \end{array} \right) \,\, .
\]  Hence the matrices are not even coarsely equivalent. 
\end{remark}

We next want to reduce the general discussion to a subclass of expansive matrices, those having only positive eigenvalues and a fixed determinant. This requires a few further auxiliary notions:
\begin{definition}
The exponential map $\mathbb{R}^{d \times d} \to {\rm GL}(d,\mathbb{R})$ is defined by
\[
\exp(X) = \sum_{k=0}^\infty \frac{X^k}{k !}\,\,.
\]
\end{definition}

It is not hard to see that the exponential map is well-defined on $\mathbb{R} ^{d \times d}$ and satisfies $\exp(X+Y) = \exp(X) \exp(Y)$, whenever $X$ and $Y$ commute \cite[Proposition 3.2.1]{HiNe}. Furthermore, for any fixed matrix $X$ the map $t \mapsto \exp(tX)$ is a continuous homomorphism $\mathbb{R} \to {\rm GL}(d,\mathbb{R})$ \cite[Theorem 3.2.6]{HiNe}, and its image is a so-called one-parameter subgroup. The next lemma states that two expansive matrices contained in the same one-parameter subgroup have equivalent quasi-norms.
\begin{lemma}
Let $A$, $B$ be expansive matrices, and assume that $A = \exp(tX)$ and $B = \exp(sX)$, for some matrix $X$ and $s,t>0$. Then $A$ and $B$ are equivalent.
\end{lemma}
\begin{proof}
The well-known formula $det(\exp(X)) = \exp(tr(X))$, with $tr(X)$ denoting the trace of the matrix $X$,  allows to compute
\[
\epsilon(A,B) =  \frac{\ln (|{\rm det}(A)|)}{\ln (|{\rm det}(B)|)} = \frac{t}{s}\,\,. 
\] and since $t \mapsto \exp(tX)$ is a homomorphism, we get
\[
A^{-k} B^{\lfloor \epsilon k \rfloor} = \exp\left( (-kt + \lfloor \frac{t}{s} k \rfloor s) X \right) = \exp(r_k X)\, 
\] with $-s \le r_k \le 0$. It follows that
\[
\sup_{k \in \mathbb{Z}} \| A^{-k} B^{\lfloor \epsilon k \rfloor} \| \le \sup_{-s \le r \le 0} \| \exp(rX) \| < \infty\,\,
\] hence $A$ and $B$ are equivalent. 
\end{proof}

A further step towards simplification is the observation that every expansive matrix is equivalent to a matrix having only positive eigenvalues. 

\begin{lemma} \label{lem:ex_pos_spec}
Let $A$ denote an expansive matrix. Then there exists a matrix $B$ which is equivalent to $A$, has only positive eigenvalues, and fulfills $|\det(A)| = \det(B)$.
\end{lemma}
\begin{proof}
First assume that $A$ has only one eigenvalue. If that eigenvalue is negative, then $B=-A$ is as desired. Now assume that the eigenvalue $\lambda$ is non-real, with $|\lambda|>1$. Given any complex number $z$, define the two-by-two matrix
\[
M_z = \left( \begin{array}{cc} {\rm Re}(z) & {\rm Im}(z) \\ - {\rm Im}(z) & {\rm Re}(z) \end{array} \right) \,\,.
\]
There exists a $C \in {\rm GL}(d,\mathbb{R})$ bringing $A$ into real Jordan normal form, i.e. such that 
\[
CAC^{-1} = \left( \begin{array}{ccccccc} M_\lambda & M_{z_1} &  &  &  & \ldots &  \\
  & M_\lambda & M_{z_2} &  &  &  & \\
  &  & \ddots & \ddots  & & & \\
  &  & & \ddots & \ddots & & \\ 
  &  & & & \ddots & \ddots & \\
 &  &             &            & & M_\lambda &  M_{z_{d/2-1}} \\ 
& & & & & & M_\lambda \end{array} \right) \,\,,
\] with $z_1,z_2 \ldots \in \{ 0 ,1 \} \subset \mathbb{C}$. 
Write $\lambda = r w$, with $r>0$ and $|w| = 1$. Then we can factor $CAC^{-1}$ as 
\[
CAC^{-1} = \left( \begin{array}{ccccccc} M_w &  &  & & & &  \\
  & M_w & & & & &  \\
  &  & \ddots & & & & \\
  &  & & \ddots & & & \\ 
  &  & & & \ddots & & \\
  &  &             &            & & M_w &  \\ 
& & & & & & M_w \end{array} \right) 
\left( \begin{array}{ccccccc} M_r & M_{\overline{w} z_1} &  &  &  & &  \\
  & M_r & M_{\overline{w} z_2} & &  &  &  \\
  &  & \ddots & \ddots  & & & \\
  &  & & \ddots & \ddots & & \\ 
 &  & & & \ddots & \ddots & \\
  &  &             &            & & M_r &  M_{\overline{w} z_{d/2-1}} \\ 
& & & & & & M_r \end{array} \right) \,\,
\] and the two factors commute. Call the factors on the right-hand side $D_1, D_2$, and let
$B = C^{-1} D_2 C$, then $B$ has $r = |\lambda|$ as only eigenvalue, in particular $|{\rm det}(A)| = {\rm det}(B)$ holds. Furthermore, the fact that $D_1$ and $D_2$ commute allows to compute, for all $k \in \mathbb{Z}$, 
\[
\| A^{-k} B^k \|  = \| C^{-1} D_1^{-k} C\| \le \| C \| \, \| C^{-1}\|\,\,
\] since $D_1$ is an orthogonal matrix. Hence Lemma \ref{lem:norm_equiv} yields the desired statement.

In the general case, we decompose $A$ into real Jordan blocks $A_1,\ldots,A_k$, and apply the above procedure to each of them. We then obtain a matrix $B$ with the desired properties via Remark \ref{rem:equiv_JNF}.
\end{proof}

One of the advantages of matrices with positive eigenvalues is that one always finds a one-parameter subgroup of ${\rm GL}(d,\mathbb{R})$ going through them. 
\begin{lemma} \label{lem:ex_one_par}
Let $A$ denote an expansive matrix with positive eigenvalues. Then there exists a matrix $X$ with 
$A = \exp(X)$. In particular, for any $c>1$ there exists an expansive matrix $B$ that is equivalent to $A$, only has positive eigenvalues, and fulfills ${\rm det}(B)=c$. 
\end{lemma}
\begin{proof}
Since the exponential of a block diagonal matrix is again block diagonal, and since $\exp(CXC^{-1}) = C \exp(X) C^{-1}$, we may assume w.l.o.g. that $A$ is a single Jordan matrix, i.e.,
\[
A = \lambda I_d + T
\] where $T$ is a strictly upper triangular matrix. Instead of this additive decomposition, we may also write $A$ as the product
\[ 
A = (\lambda  I_d) \cdot C 
\] where $C = I_d + \lambda^{-1} T$ is a unipotent matrix. In particular, $C-I_d$ is {\em nilpotent}, which means that $(C-I_d)^d = 0$. 
 Letting $T(d,\mathbb{R})$ denote the set of all unipotent matrices and $\mathfrak{t}(d,\mathbb{R})$ for the space of all nilpotent matrices, \cite[Theorem 3.3.3]{HiNe} states that 
$
\exp: \mathfrak{t}(d,\mathbb{R}) \to T(d,\mathbb{R})$ is bijective. Hence there exists $Y \in \mathfrak{t}(d,\mathbb{R})$ with $\exp(Y) = C$.
Now the fact that $Y$ and $\ln (\lambda) I_d$ commute allows to conclude that
\[
\exp(\ln (\lambda) \cdot I_d + Y) = \exp(\ln (\lambda) I_d) \exp(Y) = A\,,
\] hence $X = \ln(\lambda) \cdot I_d + Y$ is as desired.

Returning to the general case of multiple Jordan blocks, if $A = \exp(X)$, then the eigenvalues of $\exp(tX)$ are just $\lambda^t$, with $\lambda$ an eigenvalue of $A$. In particular, $\exp(tX)$ is expansive, for all $t>0$. Furthermore, $\det(A) = \exp(tr(X))>1$. Hence $tr(X)>0$, and for $t = \frac{\ln(c)}{tr(X)}$ we get $\det(\exp(tX)) = c$. Thus $B = \exp(tX)$ is as desired. 
\end{proof}

We can now show the main result of this section. 
\begin{theorem} \label{thm:class_equiv_matr}
Let $A,B$ denote expansive matrices having only positive eigenvalues, and fulfilling ${\rm det}(A) = {\rm det}(B)$.
\begin{enumerate}
\item[(a)] $A$ and $B$ are equivalent if and only if $A=B$.
\item[(b)] Let  $\lambda_1 > \lambda_2 > \ldots > \lambda_k$ denote the distinct eigenvalues of $A$, and assume that $A$ has the form
\[
A =  \left( \begin{array}{cccc} J_1 & & & \\ & J_2 & & \\ & & \ddots & \\ & & & J_k \end{array} \right)\,\,,
\] such that,
\[
 \forall 1 \le i \le k~:~(J_i-\lambda_i I_{d_i})^d = 0~. 
\]
Then $A$ and $B$ are coarsely equivalent if and only if 
\[
B =  \left( \begin{array}{cccc} J_1 & \ast &  \ast & \ast  \\ & J_2 & \ast  & \ast \\ & & \ddots & \ast \\ & & & J_k \end{array} \right)\,\,,
\] i.e., $B$ has the same blocks on the diagonal, and arbitrary entries above these blocks. 
\end{enumerate}
\end{theorem}

\begin{proof}
We first consider the case where $A$ and $B$ have only one eigenvalue $\lambda>0$, and $A$ and $B$ are coarsely equivalent. We want to show that $A=B$. 
By assumption on the spectra, we can write 
\[
 B = \lambda (I_d + N_B)
\] with $N_B^d =  0$. It follows that, for all $k$,
\begin{eqnarray*}
B^k &  = & \lambda^k \sum_{\ell=0}^k \left( \begin{array}{c} k \\ \ell \end{array} \right) N_B^\ell \\
& = & \lambda^k \sum_{\ell=0}^d \left( \begin{array}{c} k \\ \ell \end{array} \right) N_B^\ell \\
& = & \lambda^k P_k
\end{eqnarray*}
where $P_k$ is a matrix whose entries depend polynomially on $k$.
Similarly, we have
\[
A^{-1} = \lambda (I_d + N_A)
\] and thus
\[
A^{-k}  = \lambda^{-k} Q_k\,\,,
\] where the entries of $Q_k$ are polynomials in $k$. 
Assuming that $\rho_A$ and $\rho_B$ are coarsely equivalent, we have
\[
\infty > \sup_{k \in \mathbb{N}} \| A^{-k} B^k \| = \sup_{k \in \mathbb{N}} \| Q_k P_k \| \,\,.
\] The entries of $Q_k P_k$ are polynomials in $k$, and bounded. It follows that the map $k \mapsto Q_k P_k$ is {\em constant}. In particular, $A^{-2} B^2  = A^{-1} B^1$, which implies $A=B$. 

For part (a), assume that $A$ and $B$ are equivalent. W.l.o.g. $A$ is in Jordan normal form,
\[  A  = \left( \begin{array}{cccc} J_1 & & &    \\  & J_2 & & \\ & & \ddots & \\ & & & J_k \end{array} \right) \,\, 
\]
The $i$th block of $A$ corresponds to the generalized eigenspace ${\rm Ker}(A-\lambda_i I_d)^d$
associated to the $i$th eigenvalue $\lambda_i$, and since the spectra of $A,B$ are both positive, we have
by Lemma \ref{lem:char_eigenspaces}
\[
{\rm Ker}(A-\lambda_i I_d)^d = {\rm Ker}(B-\lambda_i I_d)^d
\]
But this means that $B$ has the same block diagonal structure, 
\[
 B  = \left( \begin{array}{cccc} B_1 & & &    \\  & B_2 & & \\ & & \ddots & \\ & & & B_k \end{array} \right)\,\,
\] and $(B_i - \lambda_i I_{d_i})^d = 0$. In particular, $\epsilon(J_i,B_i) = 1 = \epsilon(A,B)$, for all $i=1,\ldots,k$. Now Remark \ref{rem:equiv_JNF} yields that all pairs $J_i,B_i$ must be equivalent, and the single eigenvalue case we considered first then entails $A=B$. 
 
For part (b), we proceed by induction over the number $k$ of distinct eigenvalues, noting that the case $k=1$ has already been taken care of in the beginning of this proof. Hence it remains to prove the induction step, and we assume $k \ge 2$. 
Define 
\[
E_1 = {\rm Ker}(A-\lambda_1 I_d)^d \,\,,
\] the generalized eigenspace associated to the eigenvalue $\lambda_1$. For any matrix $C \in \mathbb{R}^{d \times d}$, one has
\[
{\rm Ker}(C^{-1}-\lambda^{-1} I_d)^d = {\rm Ker}(C-\lambda I_d)^d
\]
and hence we obtain from the assumption $\lambda_1^{-1} < \lambda_2^{-1} < \ldots$ that 
\[
E_1 = E(A^{-1}, \lambda_1^{-1},d)\,\,.
\] Now assume that $A$ and $B$ are coarsely equivalent. By choice of $\lambda_1$, and using that the spectra of both $A$ and $B$ are real, we obtain from Lemma \ref{lem:char_eigenspaces}(a) that 
\[
E_1 = E(B^{-1},\lambda_1^{-1},d) = {\rm Ker}(B-\lambda_1 I_d)^d\,\,.
\]
Thus we have derived that
\begin{equation} \label{eqn:block_AB}
A = \left( \begin{array}{cc} A_1 & 0 \\ 0 & A_2 \end{array} \right)\,\,  , \,\, B = \left( \begin{array}{cc} B_1 & C \\ 0 & B_2 \end{array} \right)
\end{equation} where $A_1=J_1$, $A_2$ is the matrix containing the remaining blocks of $A$, $B_1$ is a matrix with the single eigenvalue $\lambda_1$, and $B_2, C$ are not further specified at this point. 
Furthermore, we have $\det(A_1) = \det(B_1)$, and thus also $\det(A_2) = \det(A)/\det(A_1) = \det(B)/\det(B_1) = \det(B_2)$.

Now, by induction over $k \in \mathbb{N}$, we find that 
\begin{equation} \label{eqn:powers_B}
B^k = \left( \begin{array}{cc} B_1^k & C_k \\ 0 & B_2^k \end{array} \right) \,\,, C_k = \sum_{\ell=0}^{k-1} B_1^{\ell} C B_2^{k-1-\ell}\,\, ,
\end{equation}
which leads to
\begin{equation}
A^{-k} B^k = \left(  \begin{array}{cc} A_1^{-k} B_1^k & A_1^{-k} C_k \\ 0 & A_2^{-k} B_2^k \end{array} \right) \,.
\end{equation}
The assumption that $A$ and $B$ are coarsely equivalent implies via Lemma \ref{lem:char_cequiv_matr1} that this sequence of matrices is norm-bounded. This in turn implies that $A_i,B_i$ are coarsely equivalent, in view of the observation that $\epsilon(A_i,B_i) = 1$, for $i=1,2$. But then the case considered in the beginning of this proof  implies $A_1=B_1$. Furthermore, the induction hypothesis becomes available for $B_2$,  implying that the blocks on the diagonal of $B_2$ must coincide with those of $A_2$. 

This proves the ``only-if'' part of (b). 
For the converse direction, we assume that $A$ and $B$ have the structure assumed in part (b) of the theorem, and write $A,B$ as in (\ref{eqn:block_AB}). Note that here we assume $A_1 = B_1$, so we get that
\begin{equation}
A^{-k} B^k = \left(  \begin{array}{cc} I_{d_1} & A_1^{-k} C_k \\ 0 & A_2^{-k} B_2^k \end{array} \right) \,.
\end{equation}
and the induction hypothesis yields that the family $(A_2^{-k} B_2^k)_{k \in \mathbb{N}}$ is norm-bounded. 
Pick $\lambda_+$ strictly between $\lambda_1$ and $\lambda_2$, and let $r = \lambda_+/ \lambda_1 < 1$. We can then employ Lemma \ref{lem:exp_norm} and obtain the estimate
\begin{eqnarray*}
\| A_1^{-k} C_k \| & = & \left\| \sum_{\ell=0}^{k-1} A_1^{\ell-k} C B_2^{k-\ell-1} \right\| \\
& \le & \sum_{\ell=0}^{k-1} \| A_1^{\ell-k+1} \| \| A_1^{-1} C \| \| B_2^{k-\ell-1} \| \\
& \le & C \sum_{\ell=0}^{k-1} \lambda_1^{\ell-k+1} \lambda_+^{k-\ell-1}    \| A_1^{-1} C \| \\
& = & C' \sum_{\ell=0}^{k-1} r^\ell  \le C' \sum_{\ell=0}^\infty r^\ell < \infty\,\,. 
\end{eqnarray*}
Now Lemma \ref{lem:char_cequiv_matr1} yields that $A$ and $B$ are coarsely equivalent. 
\end{proof}

\begin{remark}
 With the criteria from the theorem, we can now see easily (by applying the Theorem to $A^T,B^T$), that the matrices 
 \[
  A = \left( \begin{array}{cc} 3 & 0  \\ 0 & 2 \end{array} \right) ~,~  B = \left( \begin{array}{cc} 3 & 0  \\ 1 & 2 \end{array} \right) 
 \] fulfill $A \sim_B B$ and $A \not\sim_{\dot{B}} B$, yielding an example that the converse of the final statement of Corollary \ref{cor:char_equiv} is false. Furthermore, we have $A^T \not\sim_{B} B^T$ even though $A \sim_B B$, in contrast to the case of homogeneous Besov spaces, see Remark \ref{rem:transpose}. 
 
 Note however that if $A$ and $B$ are both diagomal with positive entries, then $A \sim_{\dot{B}} B$ is equivalent $A \sim_B B$, and both statements are equivalent to $A = B^\epsilon$, for some positive number $\epsilon$.
\end{remark}

\begin{remark}
We call a matrix $A'$ {\bf in expansive normal form} if it is expansive, with positive eigenvalues and ${\rm det}(A') =2$. Then Lemmas \ref{lem:ex_pos_spec} and \ref{lem:ex_one_par} shows that for each expansive matrix $A$ there exists a matrix $A'$ that is in expansive normal form, and equivalent to $A$. Furthermore, Theorem \ref{thm:class_equiv_matr} ensures that $A'$ is {\bf unique}, a fact which justifies calling $A'$ {\bf the expansive normal form of $A$}. 

These observations entail that two expansive matrices $A$ and $B$ are equivalent if and only if their expansive normal forms coincide. One can break down the computation of normal forms as follows:
 \begin{itemize}
  \item Compute $A_1$ such that $|\det(A)| = \det(A_1)$, $A_1$ has only positive eigenvalue, and $A_1$ is equivalent to $A$; see Lemma \ref{lem:ex_pos_spec}.
  \item Compute $X = \log(A_1)$; see proof of Lemma  \ref{lem:ex_one_par}.
  \item Compute $A' = \exp(t X)$, for $t = \frac{\ln(2)}{\ln (|{\rm det}(A)|)}$.
 \end{itemize}
Note that in principle the different steps can be carried out using finitely many operations, if taking exponentials and logarithms of scalars are admissible, and the eigenvalues of $A$ are known: 
Determining the real Jordan normal form of $A$ amounts to computing a matrix $C$ such that $A=CJC^{-1}$, with $J$ a matrix in real Jordan normal form. Given the eigenvalues of $A$, the matrix $C$ is found by solving systems of linear equations. Then $A_1 = C M J C^{-1}$, with an easily computable (block) diagonal matrix $M$ ensuring that the product $M J$ has only positive eigenvalues (see proof of Lemma \ref{lem:ex_pos_spec}). $\log(A_1)$ is then obtained as $C\log(M J)C^{-1}$, which can again be carried out for each diagonal block separately, and quite efficiently: By \cite[Theorem 3.3.3]{HiNe}, the inverse of $\exp$ on the group of unipotent matrices is computed by the logarithm power series
 \[
  \log(I_d + Y) = \sum_{n=1}^\infty \frac{(-1)^{n+1}}{n} Y^n~,
 \] which breaks off after finitely many terms since $Y$ is nilpotent. Hence $X$ is computable from $A_1$ in finitely many steps.
 Finally, $A' = C \exp(t M J) C^{-1}$, which amounts to exponentiating each diagonal block. The latter can be done efficiently by exponentiating the eigenvalues and nilpotent parts of each block separately and then taking the products. Since  the exponential series of a nilpotent matrix again breaks off after finitely many terms, this step also only requires finitely many operations.
 
 Thus the decision whether $A$ and $B$ are equivalent is decidable in finitely many steps, and the same is true for coarse equivalence. Note that the matrix $\exp(t M J)$ arising in the computation of $A'$ has the block structure required  in part (b) of Theorem \ref{thm:class_equiv_matr}. Hence it remains to compute $B'$, and check whether $C^{-1} B' C - \exp(t \log(M J))$ vanishes on and below the block diagonal of $\exp(t \log(M J))$, where $C$ was the matrix effecting the real Jordan normal form of $A$. 
 
 Finally, recall that for the decision whether the Besov spaces associated to $A$ and $B$ resp. coincide, one needs to apply the above procedure to $A^T,B^T$, and that this distinction only matters for the inhomogeneous spaces. 
 \end{remark}

%
%

\bibliography{aib.bib}
\bibliographystyle{plain}
\end{document}